\documentclass[letterpaper,10pt,twocolumn]{IEEEtran}
\usepackage[T1]{fontenc}
\usepackage{amssymb,amsmath,amsthm,epsfig}
\usepackage[latin5]{inputenc}
\usepackage{graphicx,times,latexsym}
\usepackage{subfigure}
\usepackage{psfrag,cite}
\usepackage{setspace}
\usepackage{algorithmic, algorithm}
\usepackage{color}
\usepackage{multirow}
\DeclareMathOperator*{\Tr}{\operatorname{Tr}}

\newcommand{\norm}[1]{\| #1 \|}

\newtheorem{theorem}{Theorem}
\newtheorem{lemma}{Lemma}
\newtheorem{remark}{Remark}

\newtheorem{corollary}{Corollary}

\newtheorem{ansatz}{Assumption}
\newcommand{\oprocendsymbol}{\hbox{$\bullet$}}
\newcommand{\oprocend}{\relax\ifmmode\else\unskip\hfill\fi\oprocendsymbol}

\newcommand{\red}[1]{{\color{red} #1}}
\renewcommand{\red}[1]{{#1}}
\newcommand{\blue}[1]{{\color{black} #1}}
\renewcommand{\baselinestretch}{.98}

\definecolor{dark-green}{rgb}{0,.4,0}
\newcommand{\green}[1]{{\color{black} #1}}

\newcommand{\floor}[1]{\lfloor #1 \rfloor}
\newcommand{\ceil}[1]{\lceil #1 \rceil}
\newcommand{\modulo}{\text{mod}}

\newcommand{\real}{\mathbb{R}}
\newcommand{\integers}{\mathbb{Z}}
\newcommand{\integerspos}{\mathbb{N}}

\renewcommand{\theenumi}{(\roman{enumi})}

\allowdisplaybreaks

\graphicspath{{epsfiles/}}

\title{The value of timing information\\
	in  event-triggered control}

\author{Mohammad Javad Khojasteh
\; Pavankumar Tallapragada \; Jorge
  Cort{\'e}s \; Massimo Franceschetti
  \thanks{Preliminary results of this paper appeared in the Proceedings of the Allerton Conference on
    Communications, Control, and Computing~\cite{MJK-PT-JC-MF:16-allerton} and the IEEE Conference
    on Decision and Control~\cite{khojasteh2017time}. \newline
    M. J. Khojasteh and M. Franceschetti are with the Department of
    Electrical and Computer Engineering of University of California,
    San Diego. P.  Tallapragada is with the Department of Electrical
    Engineering, Indian Institute of Science, Bengaluru,
    India. J.~Cort{\'e}s is with the Department of Mechanical and
    Aerospace Engineering, University of California, San Diego.
    \{mkhojasteh,massimo,cortes\}@ucsd.edu, pavant@iisc.ac.in}}

\begin{document}

%\IEEEpubid{\makebox[\columnwidth]{978-1-4673-5563-6/13/\$31.00 Ğ2013 IEEE\hfill }
%\hspace{\columnsep}\makebox[\columnwidth]{}}

\maketitle

\begin{abstract}
  We study event-triggered control for stabilization of unstable
  linear plants over rate-limited communication channels subject to
  unknown, bounded delay. On one hand, the timing of event triggering
  carries implicit information about the state of the plant. On the
  other hand, the delay in the communication channel causes
  information loss, as it makes the state information available at the
  controller out of date.  Combining these two effects, we show a
  \emph{phase transition} behavior in the transmission rate required
  for stabilization using a given event-triggering strategy.  For
  small values of the delay, the timing information carried by the
  triggering events is substantial, and the system can be stabilized
  with any positive rate. When the delay exceeds a critical threshold,
  the timing information alone is not enough to achieve stabilization,
  and the required rate grows. 
 % \color{red}
  When the delay equals the inverse of
  the \emph{entropy rate} of the plant, the implicit information
carried by the triggering events perfectly compensates the loss of
information due to the communication delay,
 and we recover the  rate
  requirement prescribed by the \emph{data-rate theorem}. 
  We also provide
  an explicit construction yielding a sufficient rate for
  stabilization, as well as results for vector systems. Our
  results do not rely on any a priori probabilistic model for the delay
  or the initial conditions.
\end{abstract}


\section{Introduction}\label{sec:intro}
Cyber-physical systems (CPS) are engineering systems that integrate
computing, communication, and control. They arise in a wide range of
areas such as robotics, energy, civil infrastructure, manufacturing,
and transportation~\cite{kumar,murray2003future}. Due to
the need for tight integration of different components, requirements
and time scales, the modeling, analysis, and design of CPS present new challenges.  One key aspect  
is the presence of finite-rate, digital communication channels
in the feedback loop.  Data-rate theorems quantify the effect that
communication has on stabilization by stating that the communication
rate available in the feedback loop should be at least as large as the
\emph{intrinsic entropy rate} of the system (corresponding to the sum
of the logarithms of the unstable modes). In this way, the controller
can compensate for the expansion of the state occurring during the
communication process. Early formulations of data-rate theorems
appeared in~\cite{Delchamps,wong1999systems,baillieul1999feedback},
followed by the key contributions in
\cite{Mitter,nair2004stabilizability}. More recent extensions include
time-varying rate, Markovian, erasure, additive white and
colored Gaussian, and multiplicative noise feedback communication
channels~\cite{martins2006feedback,Paolo,Lorenzo,sukhavasi2016linear,middleton2009feedback,ardestanizadeh2012control,ding2016multiplicative},
formulations for nonlinear
systems~\cite{de2005n,liberzon2009nonlinear, topological}, 
\red{for optimal control~\cite{tatikonda2004stochastic,kostina2016ratemm,toli}},
for systems
with random
parameters~\cite{ling2010necessary,ranade2015control,nair2002communication},
and for switching
systems~\cite{liberzon2014finite,yang2016finite}. Connections
with information theory are highlighted
in~\cite{topological,sahai2006necessity,
  matveev2009estimation,minero2017anytime,girish2013}. Extended
surveys of the literature appear in~\cite{Massimo,Nair} and in the
book \cite{Yukselbook}.
 
Another key aspect of CPS  to which we pay special attention
here is the need to efficiently use the available
resources. Event-triggering control
techniques~\cite{astrom2002comparison,Tabuada,WPMHH-KHJ-PT:12} have
emerged as a way of trading computation and decision-making for other
services, such as communication, sensing, and actuation. In the
context of communication, event-triggered control seeks to prescribe
information exchange between the controller and the plant in an
opportunistic manner. In this way, communication occurs only when
needed for the task at hand (e.g., stabilization, tracking), and the
primary focus is on minimizing the number of transmissions while
guaranteeing the control objectives and the feasibility of the
resulting real-time implementation. While the
majority of this literature relies on the assumption of continuous
availability and infinite precision of the communication
channel, 
recent works also explore event-triggered implementations in the presence of
  data-rate
constraints~\cite{PT-JC:16-tac,Level,ling2016bit,pearson2017control,li2012stabilizing2,li2012stabilizingF},
and packet
drops~\cite{quevedo2014stochastic,demirel2013trade,tallapragada2016event}.
In this context, one important observation raised in~\cite{Level} is that using
event-triggering it is possible to ``beat'' the data-rate theorem.
Namely, if the channel does not introduce any delay and the controller
knows the triggering mechanism, then an event-triggering strategy can
achieve stabilization for any positive rate of transmission. This
apparent contradiction can be explained by noting that the timing of the
triggering events carries information, revealing the state of the
system. When communication occurs without delay, the controller can
track the state with arbitrary precision, and transmitting a single \blue{data payload}
bit at every triggering event is enough to compute the appropriate
control action. The works~\cite{Level,ling2016bit} take advantage of
this observation to show that any positive rate of transmission is
sufficient for stabilization when the delay is sufficiently small.
\red{In contrast, the work in~\cite{PT-JC:16-tac} studies the problem of stabilization   using an event-triggered strategy, but it does not exploit the implicit timing information carried by  the triggering events. The recent work in~\cite{linsenmayer2017delay} 
studies the required information transmission rate for containability~\cite{wong1999systems} of scalar systems, when the delay in the communication channel  is at most the inverse of the intrinsic system's entropy rate. 
Finally, ~\cite{khojasteh2017time} compares the results presented here with those of a time-triggered implementation.

The main contribution of this paper is the precise quantification of the amount of information implicit in the timing of the triggering events across the whole spectrum of possible communication delay values, and the use of \blue{both timing information and data payload} for stabilization.} For a given event-triggering strategy, we derive necessary and sufficient conditions for the
exponential convergence of the state estimation error and the
stabilization of the plant,  revealing a \emph{phase transition}
behavior of the transmission rate as a function of the delay.  Key to our
analysis is the distinction between the \emph{information access
  rate}, \blue{that is the rate at which the controller needs to receive information, conveyed by both data payload and timing information and regulated by the classic data-rate theorem;} and the
\emph{information transmission rate}, that is the rate at which the
sensor needs to send \blue{data payload}, that is affected by channel delays, as well
as design choices such as event-triggering or time-triggering
strategies. 
We show that for sufficiently low values of the delay,
the timing information carried by the triggering events is large
enough and the system can be stabilized with any positive information
transmission rate. At a critical value of the delay, the timing
information carried by the triggering events is not enough for
stabilization, and the required information transmission rate begins
to grow. When the delay reaches the inverse of the entropy rate of the
plant, the timing information becomes completely obsolete, and the
required information transmission rate becomes larger than the
information access rate imposed by the data-rate theorem. 
We also
provide necessary conditions on the information access rate for
asymptotic stabilizability and observability with exponential
convergence guarantees; necessary conditions on the information
transmission rate for asymptotic observability with exponential
convergence guarantees; as well as a sufficient condition with the
same asymptotic behavior.  We consider both scalar and vector linear
systems without disturbances.
Extensions for future work include the consideration of disturbances
and the analysis under triggering strategies different from the one
considered~here.

\subsubsection*{Notation}
Let $\real$, $\integers$ and $\integerspos$ denote the set of real
numbers, integers, and positive integers, respectively.
%We let $m(.)$ denote the Lebesgue measure on $\real$, and 
We denote by $\mathcal{B}(r)$ the ball centered at $0$ of
radius~$r$. \green{We let $\log$ and $\ln$ denote the logarithm with bases
$2$ and $e$, respectively.} For a function $f : \real \rightarrow
\real^n$ and $t \in \real$, we let $f(t^+)$ denote the limit from the
right, namely $\lim_{s \downarrow t} f(s)$.
We let $M_{n,m}(\real)$
be the set of $n \times m$ matrices over the field of real
numbers. 
%Let $0_{n}$ be the vector of size $n$ whose entries are all $0$. 
Given $A=[a_{i,j}]_{1\le i,j \le n} \in M_{n,n}(\real)$, we
let $\Tr(A)=\sum_{i=1}^{n} a_{ii}$ and $\det(A)$ denote its trace and
determinant, respectively. 
%Note that $\det(e^A)$ is equal to
%$e^{\Tr(A)}$. 
We let $m$ denote the Lebesgue measure on $\mathbb{\real}^{n}$, which
for $n=2$ and $n=3$ can be interpreted as area and volume,
respectively.
% Note that for $A \in M_{n,n}(\real)$ and $X \in \real^{n}$,
% $m(AX)=|\det(A)|m(X)$.
We let $\floor{x}$ denote the greatest integer less than or equal to
$x$, and $\ceil{x}$ denote the smallest integer greater than or equal
to $x$. We denote by $\modulo(x,y)$ the modulo function, whose value
is the remainder left after dividing $x$ by~$y$.  We let $\norm{x}$ be
the $L^2$ norm of $x$ in~$\real^n$. We let \green{ $\text{sign}(x)$ be $1$, $-1$, or $0$ when $x$ is positive, negative, or zero, respectively. }
\section{Problem formulation}\label{sec:setup}

Here we describe the system evolution, the model for the communication
channel, and the event-triggering strategy.

\subsection{System model}
We consider the standard networked control system model composed of
the plant-sensor-channel-controller tuple depicted in
Figure~\ref{fig:system}. We start with a scalar, continuous-time,
linear time-invariant (LTI) system, and then extend the model to the
vector case.

\begin{figure}%[h]
  \centering
  \includegraphics[scale=0.33]{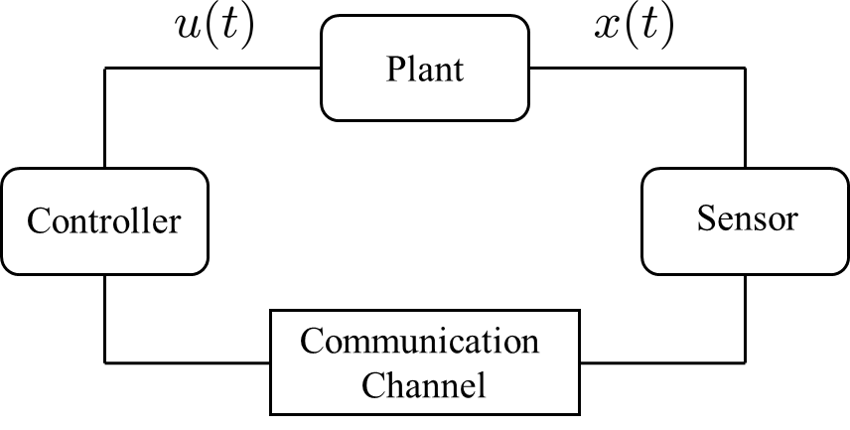}
  \caption{\red{System model. The sensor can measure the full state of the system and the controller applies the input with infinite precision and without delay. The communication channel only supports a finite rate and is subject to delay.}}\label{fig:system}
  \vspace*{-1ex}
\end{figure}

The plant dynamics are described by  
\begin{align}\label{syscon}
  \dot{x}(t)=Ax(t)+Bu(t),
\end{align}
where $x(t) \in \real$ and $u(t) \in \real$ for $t \in [0,\infty)$ are
the system state and control input, respectively. Here, $A$ is a
positive real number, 
%$B \in \real$, 
{\green{$B$ is a nonzero real number,}}
and $|x(0)|<L$ is any bounded
initial condition, where $L$ is known to both sensor and controller.
The sensor can measure the state
of the system perfectly, and the controller can apply the control input
  with infinite precision and without delay. However, the
sensor and the controller communicate through a channel that can support
only a finite communication rate and is subject to delay. 
\red{At each triggering event, the sensor can transmit a packet composed of a finite number of bits, \blue{representing a quantized version of the state}, through the communication channel, which is received by the controller entirely and without error, after an unknown, bounded delay, as described next.}
  
\subsection{Triggering strategy and controller dynamics}\label{sec:controller-dynamics}  
We denote by $\{ t_s^k \}_{k \in \integerspos}$ the sequence of times
at which the sensor transmits to the controller a packet composed of
$g(t_s^k)$ bits representing the state of the plant. For every  $k \in
\integerspos$, we let $t_c^k $ be the time at which the controller
receives the packet that the sensor transmitted at time $t_s^k$.  We
assume a uniform upper bound, known to both the sensor and the controller,
on the unknown \emph{communication delays}
% are uniformly
%upper-bounded by $\gamma$, a finite non-negative real number, namely
\begin{align}\label{gammma}
  \Delta_k= t_c^k-t_s^k \leq \gamma,  
\end{align}
and denote  the $k^{\text{th}}$ \emph{triggering interval} by
\begin{align*}
%\label{gammma111}
  \Delta'_k = t_s^{k+1}-t_s^{k}.
\end{align*}
\green{We assume the upper bound on the communication delays in \eqref{gammma} to be independent of the packet size.}
When referring to a generic triggering time or reception time, for
notational convenience we omit the superscript $k$ in $t_s^k$ and
$t_c^k$. 
\red{Our model does not assume any a priori probability
distribution for the delay, and our results hold for any random
communication delay with bounded support. }

\green{From the data received from the sensor, and from the timing at which the data is received,} the controller
maintains an estimate $\hat{x}$ of the plant state,  which  starting from $\hat{x}(t_c^{k+})$     evolves during the
inter-reception times as
%\begin{subequations}\label{eq:ET-strat}
\begin{align}\label{sysest}
  \dot{\hat{x}}(t)=A\hat{x}(t)+Bu(t), \quad t \in [t_c^k,t_c^{k+1}] .
\end{align}
The controller then computes the control input $u(t)$ based on this estimate. 
\green{The sensor can   compute the same estimate $\hat{x}(t)$ for the plant state at the controller via
  \emph{communication through the control input}~\cite{sahai2006necessity}. Namely, assuming that the input has been computed by the controller as $u(t)=\mu(\hat{x}(t))$, where $\mu$ an \textit{invertible} function known to both parties, the sensor can first compute $u(t)=(\dot{x}(t)-Ax(t))/B$ and
%, provided that $B\neq 0$. 
then compute $\hat{x}(t)$ by inversion.} 

The \emph{state estimation error} computed at the sensor is then
\begin{align*}
%\label{eq:ese}
  z(t)=x(t)-\hat{x}(t).
\end{align*}
Initially,  we let $x(0)-\hat{x}_0 = z(0)$. Without updated information from the
sensor, this error grows, and the system can potentially become
unstable.  The sensor should, therefore, select the sequence of
transmission times $\{ t_s^k \}_{k \in \integerspos}$, \green{the packet
sizes $\{g(t^k_s)\}_{k \in \integerspos}$  and the corresponding  quantization strategy used to determine the data payload,} so that the controller can
ensure stability. This choice requires a certain communication rate  available in the channel, which we wish to characterize.

To select the transmission times, we adopt an event-triggering
approach. Consider the \emph{event-triggering function} known to both
sensor and controller
%\begin{subequations}\label{eq:event-trig-strat}
\begin{align} 
v(t) = v_0 e^{-\sigma t},
\label{eq:etf}
\end{align}
%with $D$ and $\sigma$ positive real numbers, 
where $v_{0}$ and $\sigma$ are positive real numbers. A transmission
occurs whenever
\begin{align}
	|z(t)|&=v(t).
\label{eq:ets}
\end{align}
Upon reception of the packet, the controller updates the estimate of
the state according to the \textit{jump strategy}
\begin{align}
 \hat{x}(t_c^+)=\bar{z}(t_c)+\hat{x}(t_c),
 \label{eq:jumpst}
\end{align}
where $\bar{z}(t_c)$ is an estimate of ${z}(t_c)$ constructed by the
controller knowing that $|z(t_s)|=v(t_s)$, the bound~\eqref{gammma},
and the decoded packet received through the communication channel.
It follows that
\begin{align*}
%\label{eq:HLPFL}
|z(t_c^+)|=|x(t_c)-\hat{x}(t_c^+)|=|z(t_c)-\bar{z}(t_c)|.
\end{align*}
\green{
We also point out that if the control law is not invertible,  the sensor can perform the same computation of the controller to obtain $\hat{x}(t_c^{+})$,
 provided that it can   infer   the reception times   from   jumps in the control input. 

\green{By transmitting when the state estimation error $|z(t)|$ reaches the threshold $|v(t)|$}, the sensor effectively encodes  information in timing using the event-triggering rule~\eqref{eq:ets}. 
On the other hand, the data payload of the transmissions also carries information, and
 the sensor can choose any arbitrary, finite-precision quantization   of the state to construct the data payload as long as  
%In fact, to construct the packet the sensor chooses a quantization level that
it ensures that, for all $t_c \in [t_s,t_s+\gamma]$,
}
\begin{align}
|z(t_c^+)| = |z(t_c)-\bar{z}(t_c)| \le \rho(t_s):=\rho_0  e^{-\sigma \gamma} v(t_s),
  \label{eq:jump-upp}
\end{align}
where $0<\rho_0<1$ is a given design parameter. 
Note that $v(t_c)= v_0
e^{- \sigma t_c} \geq v_0 e^{-\sigma t_s} e^{- \sigma \gamma} = v(t_s)
e^{-\sigma \gamma}$, and hence~\eqref{eq:jump-upp} ensures that at
each triggering event the estimation error drops below the triggering
function, namely
\begin{align*}
|z(t_c^+)| & \le \rho_0 v(t_c).
\end{align*}
Consequently, the sequence of transmission times $\{ t_s^k \}_{k \in
  \integerspos}$ is  monotonically increasing, i.e.,
$\Delta'_k>0$ for all $k \in \integerspos$.  \blue{Moreover, based on $\dot z=Az$ and~\eqref{eq:ets}, a new transmission occurs only after the previous packet has been delivered to the controller, that is $t_s^{k+1} > t_c^k$.}
Additionally, using $\dot
z = A z$ and~\eqref{gammma}, we deduce
\begin{align}
  |z(t_c)| &\leq v(t_s) e^{A \gamma} 
  \leq v_0 e^{- \sigma (t_c - \gamma)} e^{A \gamma} \nonumber
  \\
  &= v_0 e^{(A+\sigma) \gamma} e^{-\sigma t_c}.
  \label{expodue}
\end{align}

From \eqref{eq:jump-upp} and \eqref{expodue}, it follows that the
described triggering strategy ensures an exponentially decaying
estimation error. The design parameter $\rho_0$ regulates the
resolution of the quantization, and hence the size of the transmitted
packets; as well as the magnitude of the jumps below the triggering
function, and hence the triggering rate. These also depend on the
delay, which governs the amount of overshoot of the
estimation error above the triggering function, see
Figure~\ref{fig:saw}.

\begin{figure}%[h]
  \centering
  \includegraphics[scale=0.65]{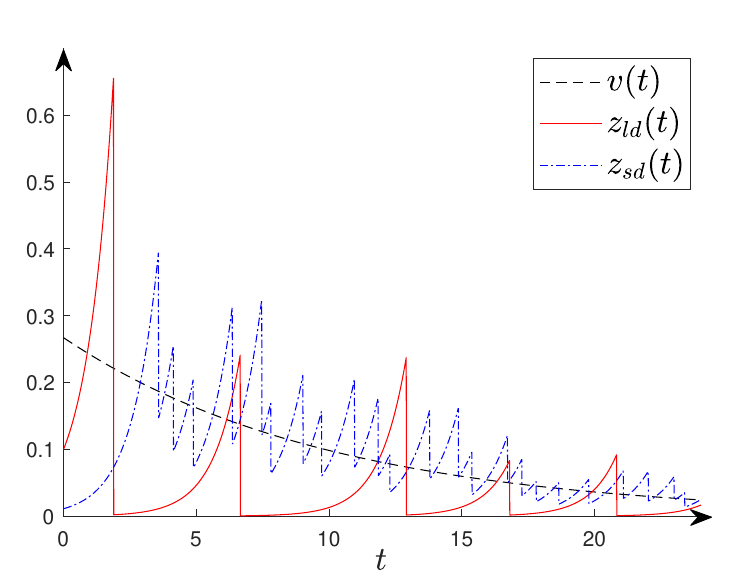}
  {\green{
  \caption{Evolution of the state estimation error $|z_{ld}(t)|$ for a larger delay upper bound $\gamma=1.2$,   and $|z_{sd}(t)|$ for a smaller delay upper bound $\gamma=0.9$.
    Here, $A=1$, $\sigma=0.1$, and $\rho_0=0.1$. The dashed exponential decaying curve represents the triggering function $v(t)=0.27e^{-\sigma t}$. A larger delay corresponds to a larger overshoot of the estimation error above the triggering function and higher uncertainty about the state at the controller. Since $\gamma$ regulates the
resolution of the quantization~\eqref{eq:jump-upp} in an exponential manner, larger delay corresponds to larger jumps under the triggering function upon reception of the packet.
    }\label{fig:saw}}}
    \vspace*{-1ex}
\end{figure}

The design parameter $v_0$ determines the initial condition of the
estimation error when the first triggering event occurs.  For any
given $0<\rho_0<1$, and $0<v_0<\infty$, our objective is to determine
the rate required to achieve these exponential bounds for all possible
delay realizations, and then provide an explicit quantization strategy
that satisfies these bounds.

\subsection{Information transmission rate}
To define the transmission rate, we take the viewpoint of the sensor
and examine the amount of information that it needs to transmit so
that the controller is able to stabilize the system.  Let $b_s(t)$ be
the number of bits \blue{in the data payload} transmitted by the sensor up to time $t$, and
define the \emph{information transmission rate} as
\begin{align*}
  % \label{Tx-rate}
  R_s = \limsup_{t \rightarrow \infty}\frac{b_s(t)}{t}.
\end{align*}
Since at every triggering time, $t_s^k$ the sensor sends $g(t_s^k)$ \blue{data payload}
bits, we have
\begin{align*}
  R_s = \limsup_{N \rightarrow \infty} \frac{\sum_{k=1}^{N} g(t_s^k)}{\sum_{k=1}^{N}
    \Delta'_k}.
\end{align*}
We now make two key observations. First, in the presence of unknown
communication delays, the state estimate received by the controller
might be out of date so that the sensor might need to send data at a
\emph{higher rate} than what is needed on a channel without delay.
Second, in the presence of event-triggered transmissions, the timing
of the triggering events carries implicit information. 
%\color{red}
For example, if
the communication channel does not introduce any delay, and assuming
that  the sensor and the controller can keep track of time with infinite precision, then the time of a triggering event reveals the system state  up to a sign, since according to~\eqref{eq:ets},
\begin{align*}
  x(t) = \hat{x}(t) \pm v(t).
\end{align*}
%\color{black}
It follows that in this case, the controller can stabilize the system
even if the sensor uses the channel very sparingly, transmitting a
single \blue{data payload} bit at a
%every 
triggering event, 
%(\blue{since the reception time of the single bit in data payload reveal the state of the system up to a sing})
that is at a much \emph{smaller
  rate} than what needed in any time-triggered implementation. 
  In
general, there is a trade-off between the information gain due to
triggering timing, and the information loss due to the delay.
% in the determination of the minimum rate that satisfies
% \eqref{eq:jump-upp} and \eqref{expodue}.
As we shall see below, this leads to a \emph{phase transition} in the
minimum rate required to satisfy~\eqref{eq:jump-upp} and as a
consequence~\eqref{expodue}.  

Finally, it is worth pointing out that the
exponential convergence of the state estimation error to zero implies
the asymptotic stabilizability of the system.

\subsection{Information access rate}
We now consider the viewpoint of the controller and examine the amount
of information that it needs to receive from the plant to be able to
stabilize the system. 
We define $b_c(t)$ to be the amount of information, measured in bits, conveyed by both  data payload and timing information, received by the controller up to time~$t$.  We define the
\emph{information access rate} as
\begin{align*}
  R_c = \limsup_{t \rightarrow \infty} \frac{b_c(t)}{t}.
%\label{infaccessrate}
\end{align*}
Classic data-rate theorems describe the information access rate
required to stabilize the system.
They are generally stated for discrete-time systems, albeit similar
results hold in continuous time as well, see
e.g.~\cite{hespanha2002towards}.  They are based on the fundamental
observation that there is an inherent entropy rate
\begin{align*}
  h = \frac{A}{\ln 2},
%\label{eq:inh}
% = \frac{\sum_{i=1}^d \lambda_i}{\ln 2}.
\end{align*}
at which the system generates information. It follows that for the
system to be stabilizable the controller must have access to state
information at a rate
\begin{align}\label{Datarate}
  R_c \ge h.
\end{align}
This result indicates what is required by the controller, and it does
not depend on the feedback structure --- including aspects such as
communication delays, information pattern at the sensor and the
controller, and whether the times at which transmissions occur are
state-dependent, as in event-triggered control, or periodic, as in
time-triggered control.

\section{Necessary condition on the access rate}\label{sec:Necessary} 

In this section, we quantify the amount of information that the
controller needs to ensure exponential convergence of the state
estimation error or the state to zero, independently of the feedback
structure used by the sensor to decide when to transmit.  The result
obtained here generalizes~\eqref{Datarate} and establishes a common
ground to compare later  against the results for the information
transmission rate, which depend on the given policy adopted by the
sensor.
The proof follows, with minor modifications, the argument
in~\cite[Propositions $3.1$ and $3.2$]{Mitter} for discrete-time
systems.

\begin{theorem}\label{thm:necc-access-rate}
  Consider the plant-sensor-channel-controller model described in
  Section~\ref{sec:setup}, with plant dynamics~\eqref{syscon}, 
  and state estimation error $z (t)$, and  let $\sigma>0$. The following necessary conditions hold:
  \begin{enumerate}
  \item If the state
    estimation error satisfies
    \begin{align*}
      |z(t)|\le |z(0)| ~e^{-\sigma t},
    \end{align*}
    then
   \begin{align}
     b_c(t) \ge t~\frac{A+\sigma}{\ln 2}+~ \log \frac{L}{|z(0)|}.
     \label{necz}
     \end{align}
   \item If the system
       is stabilizable and
   \begin{align*}
 |x(t)| \le |x(0)| ~e^{-\sigma t},
 \end{align*}
    then
      \begin{align}
     b_c(t) \ge t ~\frac{A+\sigma}{\ln 2}.
     \label{necx}
     \end{align}
  \end{enumerate}
  In both cases, the necessary information access rate is
 \begin{align}
R_c \ge \frac{A+\sigma}{\ln 2}.
\label{necrc}
\end{align}
\end{theorem}
\begin{proof}
  From~\eqref{syscon}, we have 
  \begin{subequations}
    \begin{align} 
    x(t) &= e^{At}x(0)+\alpha(t), \label{eq:auxx1} \\ 
    \alpha(t)&=e^{At} \int_{0}^{t} 
    {e^{-A\tau}Bu(\tau)d\tau}. \label{eq:auxx}
  \end{align}
\end{subequations}
Using  \eqref{eq:auxx1} we define the uncertainty set at time $t$ 
 \begin{align*} 
	\Gamma_t= \{x \in \real :
    x = e^{At}x(0)+\alpha(t) \text{ and } x(0) \in \mathcal{B}(L) \}.
  \end{align*}
  The state of the system can be any point in this uncertainty set.
  Letting $\epsilon(t)=|z(0)|~e^{- \sigma t}$, we can then find a
  lower bound on $b_c(t)$ by counting the number of one-dimensional
  balls of radius $\epsilon(t)$ that cover $\Gamma_t$. Specifically,
  \begin{align*}
    b_c(t) &\ge \log \frac{m(\Gamma_t)}{ m(\mathcal{B} (\epsilon(t)))}
    = \log \frac{e^{At}m (\mathcal{B}(L))}{2 |z(0)| ~e^{-\sigma t}}\\
    &=t ~\log e^{A+\sigma} + \log \frac{L}{|z(0)|},
  \end{align*}
  which proves (i).
 
  To prove (ii), for any given control trajectory $\{u(\tau)\}_{\tau
    \in [0,t]}$, define the set of initial conditions for which the
  plant state $x(t)$ tends to zero exponentially with rate $\sigma$,
  i.e.,
  \begin{align*}
    \Pi_{\{u(\tau)\}_{\tau \in [0,t]}}=\{x(0)\in \mathcal{B}(L):
    |x(t)|\le |x(0)|~ e^{- \sigma t} \} .
  \end{align*}
  By~\eqref{eq:auxx} $x(t)$ depends linearly on $\{u(\tau)\}_{\tau \in
    [0,t]}$, so that all the sets $\Pi_{\{u(\tau)\}_{\tau \in [0,t]}}$
  are linear transformations of each other. The measure of
  $\Pi_{\{u(\tau)=0\}_{\tau \in [0,t]}}$ is $2 |x(0)| e^{-At}
  e^{-\sigma t}$, which is upper bounded by $2 L e^{-At} e^{-\sigma
    t}$. Hence, this quantity also upper bounds the measure of each
  $\Pi_{\{u(\tau)\}_{\tau \in [0,t]}}$.  It follows that we can
  determine a lower bound for $b_c(t)$ by counting the number of sets
  of measure $2 L e^{-At} e^{-\sigma t}$ required to cover the ball
  $|x(0)| \le L$,
  and we have
  \begin{align*}
    b_c(t) & \ge \log\frac{2 L}{2 L e^{-(A+\sigma)t}} =
    t~\frac{A+\sigma}{\ln 2},
  \end{align*}
  showing (ii). 
  Finally, \eqref{necrc}   follows by dividing~\eqref{necz}
  and~\eqref{necx} by $t$ and taking the limit for $t \rightarrow
  \infty$. 
\end{proof}

\begin{remark} {\rm Theorem~\ref{thm:necc-access-rate} is
    valid for any control scheme, and the controller does not
    necessarily have to compute the state estimate
    following~\eqref{sysest}. This result can be viewed as an
    extension of the data-rate theorem with exponential convergence
    guarantees. It states that to have exponential convergence of the
    estimation error and the state, the access rate should be larger
    than the estimation entropy, the latter concept having been
    recently introduced in~\cite{liberzon2016entropy}. A similar
    result for continuous-time systems appears in~\cite{PT-JC:16-tac},
    but only for linear feedback controllers. 
\red{In fact, this work shows that the bound in~\eqref{necrc} is also sufficient for scalar systems when the controller does not use any timing information about the triggering events.}
 The classic formula of
    the data-rate theorem~\eqref{Datarate}~\cite{Mitter,nair2004stabilizability}, can be derived as a
    special case of Theorem~\ref{thm:necc-access-rate} by taking
    $\sigma \rightarrow 0$ and using
    continuity.
  }
  \oprocend
\end{remark}

\section{Necessary and sufficient conditions on the transmission
  rate}\label{sec:estim-err}

In this section, we determine necessary and sufficient conditions on
the transmission rate for the exponential convergence of the
estimation error under the event-triggered control strategy described
in Section~\ref{sec:setup}.
We start by observing that in an
event-triggering implementation, the transmission times and the packet
sizes are state-dependent.  Thus, there may be some initial conditions
and delay realizations for which both the necessary and sufficient
transmission rates are arbitrarily small. For this reason,  we provide results that
hold in worst-case conditions, namely accounting for all possible
realizations of the delay and initial conditions, without
assuming any a priori distribution on these realizations.

\subsection{Necessary condition on the transmission rate}\label{sec:Necresult}
  
Here we quantify the necessary rate at which the sensor needs to
transmit to ensure the exponential convergence of the estimation error
to zero under the given event-triggering strategy. This rate depends
on the number of bits that the sensor transmits at each triggering
event, as well as the frequency with which transmission events occur,
according to the triggering rule. Our strategy to obtain a necessary
rate consists of appropriately bounding each of these quantities.

To obtain a lower bound on the number of bits transmitted at each
triggering event, consider the uncertainty set of the sensor about the
estimation error at the controller, $z(t_c)$, given $t_s$
\begin{align*}%\label{gammatc}
  \Omega(z(t_c) | t_s)=\{y: y=\pm v(t_s)e^{A(t_c-{t}_s)}, \ t_c \in
  [t_s,t_s+\gamma]\}.
\end{align*}
On the other hand, consider the uncertainty from the point of view of
the controller about $z(t_c)$, given $t_c$
\begin{align*}%\label{unccontrl11}
  \Omega(z(t_c) |t_c) &= \{y: y=\pm v(\bar{t}_r)e^{A(t_c-\bar{t}_r)},
  \ \bar{t}_r \in [t_c-\gamma,t_c]\}.
\end{align*}
Clearly, for any $t_c \in [t_s, t_s + \gamma]$, we have $\Omega(z(t_c)
|t_c) \neq \Omega(z(t_c) |t_s)$, namely there is a \textit{mismatch}
between the uncertainties at the controller and at the sensor. The next
result shows that the uncertainty at the sensor is always smaller than
the one at the controller.

\begin{lemma}\label{inclusionsets}
  Consider the plant-sensor-channel-controller model described in
  Section~\ref{sec:setup}, with plant dynamics~\eqref{syscon},
  estimator dynamics~\eqref{sysest}, event-triggering
  function~\eqref{eq:etf}, triggering strategy~\eqref{eq:ets}, and
  jump strategy~\eqref{eq:jumpst}. Then, $\Omega(z(t_c) |t_s)
  \subseteq\Omega(z(t_c) |t_c)$.
\end{lemma}
\begin{proof}
  % Without loss of generality, we only consider the case when
  % $z(t_s)$ is positive. By~\eqref{eq:ets}, we have
  % $z(t_s)=v(t_s)$. Therefore,
  The uncertainty set of the sensor can be expressed as
  \begin{align*}
    \Omega(z(t_c) |t_s)=[v(t_s),v(t_s) e^{A \gamma}] \cup [-v(t_s) e^{A
      \gamma},-v(t_s)].
  \end{align*}
  % we have $\Omega(z(t_c) |t_c)=[v(t_c),v(t_c-\gamma) e^{A \gamma}]$.
  Noting that for any $t_c\in [t_s,t_s+\gamma]$,
  $v(\bar{t}_r)e^{A(t_c-\bar{t}_r)}$ is a decreasing function of
  $\bar{t}_r$, we have
  \begin{align*}
    \Omega(&z(t_c) |t_c)=\\
    & [v(t_c),v(t_c) e^{(A+\sigma)\gamma}] \cup [-v(t_c)
    e^{(A+\sigma)\gamma},-v(t_c)].
  \end{align*}
  The result now follows by noting that, since $v$ is a decreasing
  function, for all $t_c \in [t_s,t_s+\gamma]$ we have $v(t_s) \ge
  v(t_c)$ and $v(t_s) e^{A \gamma} \le v(t_c) e^{(A+\sigma) \gamma}$.
\end{proof}

To ensure that~\eqref{eq:jump-upp} holds, the controller needs to
reduce the state estimation error $z(t_c)$ to within an interval of
radius $\rho(t_s)$. From Lemma~\ref{inclusionsets}, this implies that
the sensor needs to cover at least the uncertainty set $\Omega(z(t_c)
| t_s)$ with one-dimensional balls of radius $\rho(t_s)$. This
observation leads us to the following lower bound on the number of
bits that the sensor must transmit at every triggering event.

\begin{lemma}\label{thm:necc-cond-ET-g}
  Under the assumptions of Lemma~\ref{inclusionsets}, if
  \eqref{eq:jump-upp} holds for all $k \in \mathbb{N}$, then the
  packet size at every triggering event must satisfy
  \begin{align}\label{g'eq}
    g(t_s^k) \ge \max \left\{0,\log \frac{(e^{A \gamma}-1)}{ \rho_0
        e^{-\sigma \gamma} }\right\}.
  \end{align}
\end{lemma}
\begin{proof}
  We compute the number of bits that must be transmitted to guarantee
  that the sensor uncertainty set $\Omega(z(t_c) | t_s)$ is covered by
  balls of radius $\rho(t_s)$. Define $\chi_\gamma=\{y: y=e^{A t}, t
  \in [0,\gamma]\}$. Since $g(t_s)$ is the packet size, it is
  non-negative. Hence, $ g(t_s) \ge \max
  \left\{0,H_{\rho(t_s)}\right\}$, where
  \begin{align}\label{entropyH}
    H_{\rho(t_s)} &:=  \log \frac{m(\Omega(z(t_c) |
      t_s))}{m(\mathcal{B}(\rho(t_s)))}  \nonumber 
    \\ 
    % &= \log \frac{m(\Omega(z(t_c) | t_s))}{2 \rho(t_s)}  \nonumber \\
    & = \log \frac{2 v(t_s) m(\chi_\gamma) }{2 \rho_0 e^{-\sigma
        \gamma} v(t_s)} \nonumber
    \\
    & = \log \frac{2 v(t_s)(e^{A \gamma}-1)}{2 \rho_0 e^{-\sigma
        \gamma} v(t_s)},
  \end{align}
  and the result follows.
\end{proof}
Our next goal is to characterize the frequency with which transmission
events are triggered. We define the triggering rate
\begin{align} \label{trate} R_{tr} &= \limsup_{N\rightarrow
    \infty}\frac{N}{\sum_{k=1}^N \Delta'_k} .
\end{align}
First, we provide an upper bound on the triggering rate that holds for
all initial conditions and
possible communication delays upper bounded by $\gamma$.

\begin{lemma}\label{Lemma1}
  Under the assumptions of Lemma~\ref{inclusionsets},
  if~\eqref{eq:jump-upp} holds for all $k \in \mathbb{N}$, then the
  triggering rate is upper bounded as
  \begin{align}\label{consecu1}
    R_{tr} \le \frac{A+\sigma}{- \ln (\rho_0 e^{-\sigma \gamma})}.
  \end{align}
\end{lemma}
\begin{proof}
  Consider two successive triggering times $t_s^k$ and $t_s^{k+1}$ and
  the reception time $t_c^k$. We have $t_s^k \leq t_c^k \leq
  t_s^{k+1}$. From~\eqref{syscon} and~\eqref{sysest}, we have
  $\dot{z}(t)=A(x(t)-\hat{x}(t))=Az(t)$.  The triggering time
  $t_s^{k+1}$ is defined by
  \begin{align}\label{needforin}
    |z(t_c^{k+}) e^{A(t_s^{k+1}-t_c^k)}|=v(t_s^{k+1}).
  \end{align}
  From~\eqref{eq:jump-upp}, we have  
  \begin{align*}
    \rho_0 e^{-\sigma \gamma} v(t_s^k) e^{A(t_s^{k+1}-t_c^k)} \ge
    v(t_s^{k+1}).
  \end{align*}
  Using \eqref{eq:etf} and   $t_s^k \le t_c^k $,  it follows that
  \begin{align*}
    \rho_0 e^{-\sigma \gamma} v_0 e^{-\sigma
      t_s^k} e^{A(t_s^{k+1}-t_s^k)} \ge v_0
    e^{-\sigma t_s^{k+1}},
  \end{align*}
  % \begin{align*}
  %   e^{A(t_s^{k+1}-t_s^k)} \ge \frac{e^{-\sigma
  %       (t_s^{k+1}-t_s^k)}}{\rho_0 e^{-\sigma \gamma}},
  % \end{align*}
  and after some algebra we obtain
  \begin{align*}
    (A+\sigma)(t_s^{k+1}-t_s^k) \ge - \ln (\rho_0 e^{-\sigma \gamma}).
  \end{align*}
 We then have the uniform lower
  bound for all $k \in \mathbb{N}$
  \begin{align}\label{lowerboundzenoebeha}
    \Delta'_k= t_s^{k+1}-t_s^k \ge  \frac{- \ln (\rho_0
      e^{-\sigma \gamma})}{A + \sigma},
  \end{align}
  which substituted into~(\ref{trate}) leads to the desired upper
  bound on the triggering rate.
\end{proof}

\red{
\begin{remark}
{\rm In addition to providing an upper bound on the
    triggering rate, Lemma~\ref{Lemma1} also shows that our
    event-triggered scheme does not exhibit ``Zeno behavior''
    \cite{johansson1999regularization}, namely the occurrence of infinitely many
    triggering events in a finite time interval. This follows from     the uniform lower bound for all $k \in \mathbb{N}$ on the size of     triggering interval in~\eqref{lowerboundzenoebeha}.} 
\oprocend
\end{remark}
}

If $\Delta_k=0$ and $|z(t_c^{k+})|=\rho_0 e^{-\sigma \gamma}
v(t_s)$ for all $k \in \integerspos$, then the upper bound on the triggering rate in Lemma~\ref{Lemma1} is tight.  \green{Our next goal is} to provide a lower bound on the triggering rate that holds for a given initial condition and delay value. To obtain \green{a nontrivial} lower bound, we need to restrict the class of allowed quantization policies \green{used to construct the data payload.}
We assume that, at each triggering event, there exists a delay such that the sensor can reduce the
estimation error at the controller to at most a fraction of the
maximum value $\rho(t_s)$ required by~\eqref{eq:jump-upp}. This is a
natural assumption, and in practice corresponds to assuming an upper
bound on the size of the packet that the sensor can transmit at every triggering event \green{and hence on the precision of the quantization strategy.} Without such a bound, a packet may carry an
unlimited amount of information, \blue{the quantization error may become arbitrary small,} and $|z(t_c^+)|$ may become
arbitrarily close to zero for all delay values, resulting in a
triggering rate arbitrarily close to zero. The next assumption precludes such an unrealistic scenario.

\red{\begin{ansatz}\label{Defnition11}
% Recall that $\bar{z}(t_c)$ is controller estimation of state estimation error at time $t_c$. 
The controller can only achieve $\nu$-precision quantization. Formally,  letting $\beta =\frac{1}{A} \ln
  (1+2 \rho_0 e^{-\sigma \gamma})$, we assume
  there exists  a delay realization $\{ \Delta_k \le \beta
  \}_{k \in \integerspos}$, an initial condition $x(0)$,  and a real number  $\nu\geq1$, such that for all $k \in
  \mathbb{N}$
  \begin{align}\label{eq:deltaz}
    |z(t_c^k)-\bar{z}(t_c^k)|\geq \frac{\rho(t_s^k)}{\nu}.
    % \eqoprocend
  \end{align} 
  
\end{ansatz}}
\smallskip
\color{black}

The upper bound $\beta$ on the delay in Assumption~\ref{Defnition11}
corresponds to the time required for the state estimation error to
grow from $z(t_s)$ to $z(t_s)+2 \rho(t_s)$. In fact,
\begin{align*}
  z(t_c) &= z(t_s) e^{A \beta} = z(t_s) (1+ 2 \rho_0 e^{-\sigma
    \gamma}),
\end{align*}
from which it follows that 
\begin{align*}
z(t_c) - z(t_s) =2 z(t_s) \rho_0
e^{-\sigma \gamma}, 
\end{align*}
and since $z(t_s) = \pm v(t_s)$, we have
\begin{align*}
  |z(t_c) - z(t_s)|=2 \rho(t_s).
\end{align*}
To ensure~\eqref{eq:jump-upp}, the size of the quantization cell
should be at most $2\rho(t_s)$. As the delay takes values in
$[0,\beta]$, the value of $z(t_c)$ sweeps an area of measure $2
\rho(t_s)$.  It follows that Assumption~\ref{Defnition11} corresponds
to the existence of a value of the communication delay for which the
uncertainty ball about the state shrinks from having a radius at most
$\rho(t_s)$ to having a radius at least $\rho(t_s)/\nu$.  With this
assumption in place, we can now compute the desired
lower bound on the triggering rate.

\begin{lemma}\label{thm:necc-cond-ET}
  Under the assumptions of Lemma~\ref{inclusionsets},
  if~\eqref{eq:jump-upp} holds with $\nu$-precision for all $k \in
  \mathbb{N}$, then there exists a delay realization $\{ \Delta_k
  \}_{k \in \integerspos}$ and an initial condition such that
  \begin{align*}
    % \label{lowerTRATE}
    R_{tr} \geq \frac{A+\sigma}{\ln\nu+\ln(2+\frac{e^{\sigma \gamma}}{\rho_0})}.
  \end{align*}
\end{lemma}
\begin{proof}   
  By Assumption~\ref{Defnition11}, for all $k \in \integerspos$ there
  exists a delay $\Delta_k \leq \beta$ such that
  \begin{align*}
    |z(t_c^{k+})|\ge (1/\nu) \rho_0 v(t_s^k) e^{-\sigma \gamma}.
  \end{align*}
  From the definition of the triggering time $t_s^{k+1}$ in~\eqref{needforin}, we also have
  \begin{align*}
    (1/\nu)\rho_0 e^{-\sigma \gamma} v(t_s^k) e^{A(t_s^{k+1}-t_s^{k}-
      \Delta_k)} \le v(t_s^{k+1}).
  \end{align*}
  Noting that for all $k \in \integerspos$, $\Delta_k \le \beta$,  we have
  \begin{align*}
    (1/\nu)\rho_0 e^{-\sigma \gamma} v(t_s^k) e^{A(t_s^{k+1}-t_s^{k}-
      \beta)} \le v(t_s^{k+1}).
  \end{align*}
  By dividing both sides by $ (1/\nu)\rho_0 e^{-\sigma \gamma}$
  and using the definition of triggering function, we obtain
  \begin{align*}
    e^{(A+\sigma)(t_s^{k+1}-t_s^{k})} \le \frac{1}{(1/\nu)\rho_0 e^{-\sigma \gamma}e^{-A \beta}} .
  \end{align*}
  Taking the logarithm, we get
  \begin{align}\label{interdelay}
    \Delta_k' = t_s^{k+1}-t_s^{k} \le \frac{-\ln((1/\nu)\rho_0
      e^{-\sigma \gamma})+A\beta}{A+\sigma}.
  \end{align}
  By substituting \eqref{interdelay} into~\eqref{trate}, we finally
  have
  \begin{align*}
    R_{tr} &\ge \lim_{N\rightarrow \infty}
    \frac{1}{\frac{-\ln((1/\nu)\rho_0 e^{-\sigma
          \gamma})}{A+\sigma}+\frac{A}{A+\sigma}\beta}
    \\
    \nonumber &=\frac{A+\sigma}{\ln \nu-\ln(\rho_0 e^{-\sigma
        \gamma})+\ln (1+2 \rho_0 e^{-\sigma \gamma})}
    \\
    \nonumber &=\frac{A+\sigma}{\ln\nu+\ln(2+\frac{e^{\sigma
          \gamma}}{\rho_0})}.
  \end{align*}
\end{proof}

We can now combine Lemma~\ref{thm:necc-cond-ET-g} and
Lemma~\ref{thm:necc-cond-ET} to obtain a lower bound on the
information transmission rate.

\begin{theorem}\label{THM:NECESAAPP}
  Under the assumptions of Lemma~\ref{inclusionsets},
  if~\eqref{eq:jump-upp} holds with $\nu$-precision for all $k \in
  \mathbb{N}$, then there exists a delay realization $\{ \Delta_k
  \}_{k \in \integerspos}$ and an initial condition such that
  \begin{align}\label{necessT}
    R_s \ge \frac{A+\sigma}{\ln\nu+\ln(2+\frac{e^{\sigma
          \gamma}}{\rho_0})}\max \left\{0,\log \frac{(e^{A
          \gamma}-1)}{ \rho_0 e^{-\sigma \gamma} }\right\}.
  \end{align}
\end{theorem}

\begin{remark}
  {\rm Theorem~\ref{THM:NECESAAPP} provides a necessary transmission
    rate for the exponential convergence of the estimation error to
    zero using our event-triggering strategy. By noting that the lower bound in~\eqref{necessT} does not depend on $v_0$, it is easy to check that
    as $\sigma \rightarrow 0$, this result also gives a necessary
    condition for asymptotic stability, although it does not provide
    an exponential convergence guarantee of the state. } \oprocend
\end{remark}

\subsection{Phase transition behavior}\label{sec:phasetr}\label{PHTRE12}

We now show a phase transition for the rate required for stabilization
expressed in Theorem \ref{THM:NECESAAPP}.  By combining
Lemmas~\ref{Lemma1} and \ref{thm:necc-cond-ET}, we have
\begin{align*}%\label{approxRt}
  \frac{A+\sigma}{\ln \nu+\ln(2+\frac{1}{\rho_0e^{-\sigma
        \gamma}})}\le R_{tr} \le \frac{A+\sigma}{- \ln (\rho_0
    e^{-\sigma \gamma })}.
\end{align*}
\color{black}
It follows that if
%$\nu = 2$
%
%and 
$\rho_0 \ll e^{\sigma \gamma}/\max\{2,\nu\}$, we can neglect the value
of 2 inside the logarithm in the left-hand side, as well as $\ln\nu$,
and we have
\begin{align*}
  % \label{approxRt3}
  R_{tr}\approx\frac{A+\sigma}{-\ln(\rho_0 e^{-\sigma \gamma})}.
\end{align*} 
In this case, the necessary condition on the transmission rate can be
approximated as
\begin{align}
  \label{approxnec}
  R_s \geq \frac{A+\sigma }{\ln 2} \max \left\{0 ,1+\frac{\log (e^{A
        \gamma}-1)}{-\log (\rho_0 e^{-\sigma \gamma})}\right\}.
\end{align}
\color{black}
We use this approximation to discuss the phase transition behavior.
The approximation clearly holds for large values of the
delay upper bound $\gamma$.  It also holds for small values of $\gamma$, since in
this case both \eqref{necessT} and~\eqref{approxnec} tend to zero. For
intermediate values of $\gamma$, the approximation holds for large
values of the convergence rate $\sigma$.  The phase transition is
illustrated in Figure~\ref{fig:ness}.

\begin{figure}[t]
  \centering
  \includegraphics[height=7.5cm,width=8.2cm]{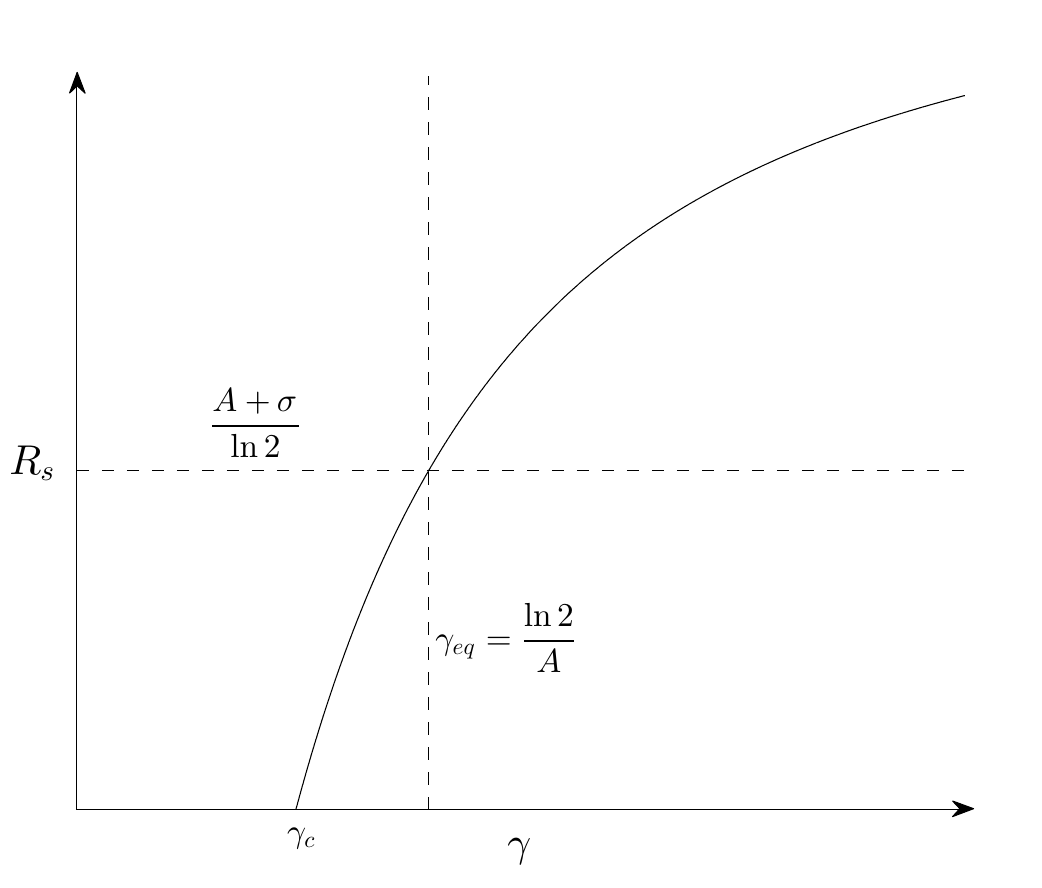}
  \caption{
  \red{Illustration of the phase transition behavior in \eqref{approxnec}. $R_s$ is measured in $\mbox{bits}/\mbox{sec}$, and $\gamma$ is measured in $\mbox{sec}$. The plot is valid for a generic  system and design parameters. In this specific example, we have chosen $A=5$, $\sigma=3$, and $\rho_0=0.7$. Consequently, $(A+\sigma)/\ln2=11.5416$, $\ln2/A=0.1386$, and $\gamma_c=0.0864$.
  }}\label{fig:ness}
  \vspace*{-1ex}
\end{figure}
\begin{figure}[t]
  \centering
  \includegraphics[height=7.5cm,width=8.2cm]{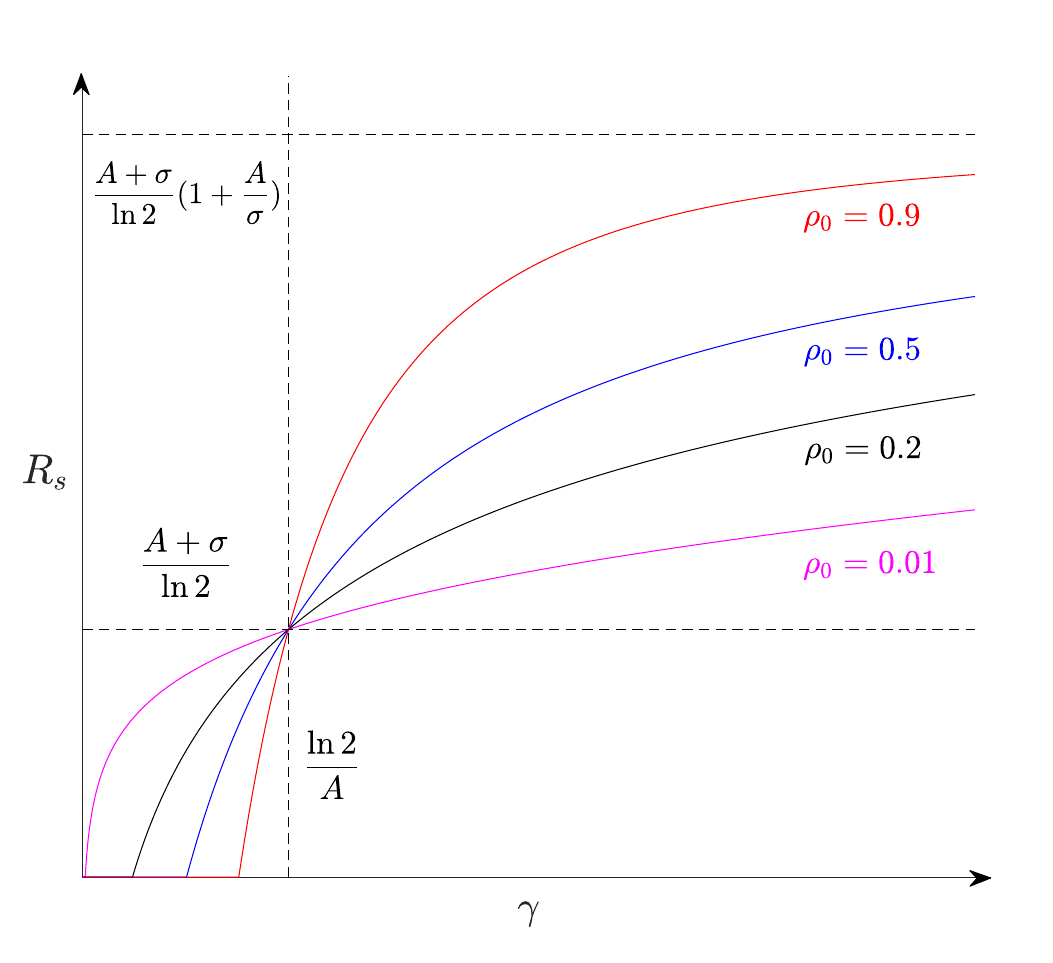}
  \caption{
  \red{Illustration of the phase transition behavior in~\eqref{approxnec}  for different values of $\rho_0$. $R_s$ is measured in $\mbox{bits}/\mbox{sec}$, and $\gamma$ is measured in $\mbox{sec}$. The plots are  valid for a generic  system and design parameters. In this  specific example, we have chosen $A=1$, and $\sigma=0.5$. Therefore, $(A+\sigma)/\ln2=2.1640$, $\ln2/A=0.6931$, $\frac{A+\sigma}{\ln2}(1+\frac{A}{\sigma})=6.4921$}.} \label{fig:difrho}
  \vspace*{-1ex}
\end{figure}

We make the following observations. For small values of~$\gamma$, the
amount of timing information carried by the triggering events is
higher than what is needed to stabilize the system and the value of
$R_s$ is zero. This means that if the delay is sufficiently small,
then only a positive transmission rate is required to track the state
of the system and the controller can successfully stabilize the system
by receiving a single bit of information at every triggering event.
This situation persists until a critical value $\gamma = \gamma_c$ is
reached. This critical value is the solution of the equation
\begin{align*}
  % \label{eq:Sol}
  e^{A \gamma} - \rho_0 e^{-\sigma \gamma}=1.
\end{align*}
For this level of delay, the timing information of the triggering
events becomes so much out of date that the transmission rate must
begin to increase.

When $\gamma$ reaches the equilibrium point $\gamma_{eq}=\ln 2/A$,
which equals the inverse of the intrinsic entropy rate of the system,
the timing information carried by the triggering events compensates
exactly the loss of information due to the delay introduced by the
communication channel. This situation is analogous to having no delay,
but also no timing information. It follows that in this case the
required transmission rate matches the access rate in
Theorem~\ref{thm:necc-access-rate}, and we have $R_s= (A+ \sigma)/ \ln
2$. 

When $\gamma$ is increased even further, then the timing information
carried by event triggering is excessively out of date and cannot
fully compensate for the channel's delay. The required transmission
rate then exceeds the access rate imposed by the data-rate theorem. In
this case, a more precise estimate of the state must be sent at every
triggering time to compensate for the larger delay.  Another
interpretation of this behavior follows by considering the definition
$H_{\rho(t_s)}$ in~\eqref{entropyH}.  The value $\gamma =
\gamma_{eq}=\ln2/A$ marks a transition point for $H_{\rho(t_s)}$ from
negative to positive values.  For $\gamma>\gamma_{eq}$ event
triggering does not supply enough information and $H_{\rho(t_s)}$
presents a positive information balance in terms of the number of bits
required to cover the uncertainty set.  On the other hand, for
$\gamma<\gamma_{eq}$, event triggering supplies more than
enough information, and $H_{\rho(t_s)}$ presents a negative
information balance. We can then think of event triggering as a
``source'' supplying information, the controller as a ``sink''
consuming information, and $H_{\rho(t_s)}$ as measuring the balance
between the two, indicating whether additional information is needed
in terms of quantized observations sent through the channel.

Finally, Figure~\ref{fig:difrho} illustrates the phase transition for
different values of $\rho_0$.  For $\gamma<\gamma_{eq}$, since
according to~\eqref{consecu1} smaller values of $\rho_0$ imply fewer
triggering events, it follows that curves associated to smaller values
of $\rho_0$ must have larger transmission rates to compensate for the
lack of timing information.  On the other hand, for
$\gamma>\gamma_{eq}$ the situation is reversed. The timing
information carried by the triggering events is now completely
exhausted by the delay, and the controller relies only on the state
information contained in the quantized packets. Since, according
to~\eqref{g'eq}, smaller values of $\rho_0$ imply larger packets sent
through the channel and, for each value of the delay, the information
in the larger packets becomes out of date at a slower rate than that
in the smaller packets, it follows that in this case curves associated
to smaller values of $\rho_0$ correspond to smaller transmission
rates.  Finally, we observe that all curves have the same asymptotic
behavior for large values of $\gamma$, which is independent of
$\rho_0$. This occurs because as $\gamma$ increases, more information
needs to be sent through the channel and also the triggering rate
decreases.  Taking both effects into account yields the asymptotic
value of the transmission
rate~$\frac{A+\sigma}{\ln2}(1+\frac{A}{\sigma})$.

\begin{remark} {\rm The value of  $\gamma_c$ is a threshold
    distinguishing whether~\eqref{necessT} is zero or strictly
    positive. This threshold tends to $\gamma_{eq} =\ln
    2/A$ as $\sigma \rightarrow 0$ and $\rho_0 \rightarrow 1$. This is
    consistent with the fact that in this case there is only an
    asymptotic convergence guarantee (not an exponential one), and
    when the delay upper bound $\gamma$ is at most the inverse of
    entropy rate of the system only a positive transmission rate is
    necessary for stabilization.
   % one recovers the  classical data-rate theorem with $R_s= A/\ln 2$.
    }  \oprocend
\end{remark}

\subsection{Sufficient condition on the transmission
  rate}\label{sec:sufresult}

We now determine a sufficient transmission rate for the exponential
convergence of the state estimation error using the event-triggering
strategy described in Section~\ref{sec:controller-dynamics}. 

In our strategy, we let the sensor send a packet
consisting of the sign of $z(t_s)$ and a quantized version of $t_s$
to the controller.  Using the bound \eqref{gammma}, and the decoded
packet, the controller constructs $q(t_s)$, a quantized version of
$t_s$.  The controller then estimates $z(t_c)$ as follows
\begin{align}\label{ctrlapp1}
   \bar{z}(t_c)=\text{sign}(z(t_s)) v(q(t_s))e^{A(t_c-q(t_s))}.
\end{align}
The next result provides a bound on the error in the time quantization
that guarantees that the requirements of the design are satisfied.

\begin{lemma}\label{lemmasu}
Under the assumptions of Lemma~\ref{inclusionsets},  using~\eqref{ctrlapp1}, if 
   \begin{align}\label{Qua}
   |t_s-q(t_s)| \le \frac{1}{A+\sigma} \ln(1+\rho_0 e^{- (\sigma+A) \gamma})
\end{align}
then~\eqref{eq:jump-upp} holds.
\end{lemma}
\begin{proof}
Using~\eqref{ctrlapp1}, it follows that
 \begin{align}\label{eq:diffzzbar}
|z(t_c)-&\bar{z}(t_c)|\\\nonumber&=v(t_s)e^{A(t_c-t_s)}\left|1-\frac{v(q(t_s))}{v(t_s)}e^{A(t_s-q(t_s))}\right|
\\\nonumber&=v(t_s)e^{A(t_c-t_s)}\left|1-\frac{v_0 e^{-\sigma q(t_s)}}{v_0 e^{-\sigma t_s}}e^{A(t_s-q(t_s))}\right|
\\\nonumber&=v(t_s)e^{A(t_c-t_s)}\left|1-e^{(A+\sigma)(t_s-q(t_s))}\right|.
\end{align}
As a consequence, \eqref{eq:jump-upp} may be expressed as
\begin{align*}
|1-e^{(A+\sigma)(t_s-q(t_s))}| \le \rho_0 e^{-\sigma \gamma} e^{-A(t_c-t_s)}.
\end{align*}
The smallest possible value  of $e^{-A(t_c-t_s)}$ for
$(t_c - t_s) \in [0, \gamma]$ is $e^{-A \gamma}$.
%This smallest possible value is corresponded to the delay equal to $\gamma$.
%We want to ensure~\eqref{eq:jump-upp} via the jump
%strategy~\eqref{eq:jumpst} for all delay values. Therefore, we need
%
Therefore, by ensuring
\begin{align}\label{eq:ts-quant-bound}
  \left|1-e^{(A+\sigma)(t_s-q(t_s))}\right| \le \rho_0 e^{- (\sigma+A) \gamma},
\end{align}
we can also ensure~\eqref{eq:jump-upp}. The condition in
~\eqref{eq:ts-quant-bound} can be rewritten as
%   \begin{align*}
%  -\rho_0 e^{- (\sigma+A)\gamma} \le 1-e^{(A+\sigma)(t_s-q(t_s))} \le \rho_0 e^{- (\sigma+A)\gamma}, 
%  \end{align*}
 \begin{align*}
 1-\rho_0 e^{- (\sigma+A) \gamma}\le e^{(A+\sigma)(t_s-q(t_s))} \le 1+\rho_0 e^{- (\sigma+A) \gamma}.
 \end{align*}
Taking logarithms and dividing by $(A+\sigma)$, we obtain
 \begin{align}\label{comp1}
 \frac{1}{A+\sigma} \ln(1-x') \le t_s-q(t_s) \le \frac{1}{A+\sigma} \ln(1+x'),
  \end{align}
 where $x'=\rho_0 e^{- (\sigma+A) \gamma}$.
It follows that to satisfy~\eqref{eq:jump-upp} for all delay values it is enough that
  \begin{align*}
  |t_s-q(t_s)| \le \min\{|\frac{1}{A+\sigma} \ln(1-x')|,|\frac{1}{A+\sigma} \ln(1+x')|\}.
  \end{align*}
The result now follows.
%     \begin{align*} 
%    |t_s-q(t_s)| \le \frac{1}{A+\sigma} \ln(1+\rho_0 e^{- (\sigma+A) \gamma}).
% \end{align*}
\end{proof}

The next result presents a sufficient transmission rate, along with the design that meets it. 
\green{
\begin{theorem}\label{thm:suf-cond-ET} 
Under the assumptions of Lemma~\ref{inclusionsets}, 
if the state estimation error satisfies  $|z(0)|<v_0$, then
for any information transmission rate
  \begin{align}\label{Sufi}
    R_s \ge
    \frac{A+\sigma}{-\ln(\rho_0 e^{-\sigma \gamma})}
    \max\left\{0,1+\log\frac{b\gamma (A+\sigma)}{\ln(1+\rho_0 e^{-
          (\sigma+A) \gamma})}\right\},
  \end{align}
  where $b>1$,
there exists a quantization policy that achieves~\eqref{eq:jump-upp} 
for all $k \in \mathbb{N}$ 
(and consequently 
$|z(t)| \le v_0 e^{(A+\sigma) \gamma} e^{-\sigma t}$). 
\end{theorem}
}
\begin{proof}
  % [Proof of Theorem \ref{thm:suf-cond-ET}]
  Our proof strategy is as follows. We design a quantizer to construct
  a packet of length $g(t_s)$ that the sensor sends to the
  controller. Using this packet, the decoder reconstructs the
  quantized version $q(t_s)$ of $t_s$ satisfying~\eqref{Qua}. The
  result then follows from Lemma~\ref{lemmasu} and quantifying the
  associated transmission rate.

  In our construction, the first bit of the packet determines the sign
  of $z(t_s)$, i.e., whether $z(t_s)=+ v(t_s)$ or
  $z(t_s)=-v(t_s)$. For quantizing $t_s$, we first divide the whole
  positive time line in sub-intervals of length $b\gamma$. Recall that
  the controller receives a packet at time $t_c$, and $t_s \in
  [t_c-\gamma,t_c]$. Noting that $b\gamma>\gamma$, upon the reception
  of the packet at time $t_c$ the decoder identifies two consecutive
  sub-intervals of length $b\gamma$ that $t_s$ can belong to --- the
  second bit of the packet is $\modulo \left( \floor{ \frac{t_s}{b
        \gamma} }, 2 \right) $, which informs the decoder that $t_s
  \in [ \iota b \gamma, (\iota+1) b \gamma]$ for some fixed
  $\iota$. The encoder divides this interval uniformly into
  $2^{g(t_s)-2}$ sub-intervals, one of which contains $t_s$. After
  receiving the packet, the decoder determines the correct
  sub-interval and chooses $q(t_s)$ as the middle point of it. With
  this strategy, we have
  \begin{align}\label{Quntrulesuff1}
    |t_s-q(t_s)| \leq\frac{b\gamma}{2^{g(t_s)-1}}.
  \end{align}
  Hence, from Lemma~\ref{lemmasu}, it is enough to ensure
  \begin{align}\label{eq:condlength1}
    \frac{b\gamma}{2^{g(t_s)-1}} \le \frac{1}{A+\sigma} \ln(1+\rho_0
    e^{- (\sigma+A) \gamma}),
  \end{align}
  to guarantee that~\eqref{eq:jump-upp} holds. This is equivalent to
  \begin{align}
    \label{lowerbndgsuf}
    g(t_s) \ge \max\left\{0,1+\log\frac{b\gamma
        (A+\sigma)}{\ln(1+\rho_0 e^{- (\sigma+A) \gamma})}\right\}.
  \end{align}
  The characterization~\eqref{Sufi} of the transmission rate now
  follows from using this bound and the uniform upper bound on the
  triggering rate~\eqref{consecu1}.
\end{proof}

Theorem~\ref{thm:suf-cond-ET} ensures the exponential convergence of
the state estimation error. The following result shows
that~\eqref{Sufi} is sufficient for asymptotic stabilizability when
employing a linear controller.

\begin{corollary}\label{Stability}
{\green{Under the assumptions of Theorem~\ref{thm:suf-cond-ET},~\eqref{Sufi} is also a
  sufficient condition for asymptotic stabilizability.}}
\end{corollary}
\begin{proof}
  With  $u(t)=-K \hat{x}(t)$, we can rewrite~\eqref{syscon} as
\begin{align*}
  \dot{x}(t)=(A-BK)x(t)+BKz(t).
\end{align*}
  As a consequence, we have
   \begin{align*}
  x(t)=e^{(A-BK)t}x(0)+e^{(A-BK)t} \int_{0}^{t}e^{-(A-BK)\tau}BKz(\tau)d\tau.
  \end{align*}
%  By using the jump strategy, whenever the controller receives a packet 
 % $u(t_c^+)$ is equal to $-K(\hat{x}(t_c)+\bar{z}(t_c))$. 
  According to Theorem~\ref{thm:suf-cond-ET},~\eqref{Sufi} is
  sufficient to guarantee $\lim_{t \rightarrow \infty}z(t)=0$.
  {\green{Since $B \neq 0$ one can choose $K$ such that $A-BK<0$,}}
  and it follows that criterion~\eqref{Sufi} is also sufficient
  for $\lim_{t \rightarrow \infty}x(t)=0$. Stability can also be
  guaranteed from the above expression.
\end{proof}

It should be clear
that if the quantization policy designed for establishing
Theorem~\ref{thm:suf-cond-ET} satisfies Assumption~\ref{Defnition11}, then
the number of bits transmitted at each triggering time is
finite. We conclude this section by providing a  condition
under which the designed policy satisfies
Assumption~\ref{Defnition11}.

\begin{theorem}\label{PROBDEF1}
  Under the assumptions of Lemma~\ref{inclusionsets}, let
  $\nu \ge 2$,  and  let the number of bits in each transmitted packet be a constant $g(t_s^k) = g$. If $g$  satisfies  the lower
  bound~\eqref{lowerbndgsuf} and the upper bound
  \begin{align}\label{gtsupperbound}
    g \le \log \frac{b \gamma (A+\sigma)}{ \left|\ln
    \left(1-\frac{1}{(\nu-1) \left(2+\dfrac{1}{\rho_0e^{-\sigma\gamma}}\right)}\right)\right| },
  \end{align}
 and 
  \begin{equation}\label{eq:expansion-cond}
    \frac{ 1 - e^{-(A+\sigma) \frac{\delta}{2}} }{ 1 - e^{-(A+\sigma)
        \frac{\delta}{4}} } \geq e^{(A+\sigma) \frac{3\delta}{4}} ,
  \end{equation}
  where $\delta =  b \gamma / 2^{g - 2} $, then the
  quantization policy used in Theorem~\ref{thm:suf-cond-ET} satisfies
  Assumption~\ref{Defnition11} at every triggering time.
\end{theorem}
\begin{proof}
The proof follows from the following two claims.

  \emph{Claim (a):} For all $k \in \integerspos$, if $t_s^k$ satisfies
  \begin{align}\label{sanitycheck}
    - \frac{ \delta }{ 2 } = - \frac{b\gamma}{2^{g-1}} \le t_s^k -
    q(t_s^k) \le - \frac{b\gamma}{2^{g}} = - \frac{ \delta }{ 4 } ,
  \end{align}
  then there exists a delay $\Delta_k \leq \beta$ such
  that~\eqref{eq:deltaz} is satisfied.

  \emph{Claim (b):} The sequence of transmission times $\{t_s^k\}$ is
  uniquely determined by the initial condition $z(0)$ and there exists
  a $z(0)$ such that for each $k \in \integerspos$, $t_s^k$
  satisfies~\eqref{sanitycheck}.

   We first
  prove Claim~(a). Note that when the sensor transmits $g$
  bits, lower bounded by~\eqref{lowerbndgsuf}, the upper bound on the
  quantization error~\eqref{Quntrulesuff1} holds and
  thus~\eqref{sanitycheck} is well defined. From~\eqref{sanitycheck}
  and~\eqref{gtsupperbound}, we have
  \begin{align}\label{eq:mjkh1}
    t_s^k - q(t_s^k) \le \frac{1}{A+\sigma} \ln \left(1 - \frac{1}{(\nu-1)
    (2+\frac{1}{\rho_0 e^{-\sigma\gamma}})}\right),
  \end{align}
  where we have used the fact that $\nu \geq 2$ to simplify the
  absolute value. We rewrite this inequality as
  \begin{align*}
    1-e^{(A+\sigma)(t_s^k-q(t_s^k))} 
    &\geq \frac{\rho_0 e^{-\sigma \gamma}}{(\nu-1) (1+2 \rho_0 e^{-\sigma
      \gamma})} > 0.
  \end{align*}
  Thus, from~\eqref{eq:diffzzbar}, we see that
  \begin{align*}
    | z(t_c^k) - \bar{z}(t_c^k) | 
    &\geq v(t_s^k) e^{A \Delta_k} \frac{\rho_0 e^{-\sigma
      \gamma}}{(\nu-1) (1+2 \rho_0 e^{-\sigma \gamma})} 
    \\
    &\geq \frac{ \rho(t_s^k) }{ \nu } e^{A (\Delta_k - \beta + \ln(
      \frac{\nu}{\nu -1} ) ) } 
    \\
    &\geq \frac{ \rho(t_s^k) }{ \nu }, \quad \forall \Delta_k \in \left[ \beta - \ln \left(
      \frac{\nu}{\nu -1} \right) , \beta \right] ,
  \end{align*}
  where in the second inequality, we have used the definition of
  $\rho(t_s^k)$ in~\eqref{eq:jump-upp}. This proves Claim~(a). 

  We now prove Claim~(b). First,  we need
  to determine the dependence of $t_s^{k+1}$ on $t_s^k$ and
  $\Delta_k$. Recall the triggering rule~\eqref{eq:ets}, which we
  express as $
    v(t_s^k) e^{-\sigma \Delta'_k}
    = | z(t_c^{k+})| e^{A(\Delta'_k - \Delta_k)}
    = v(t_s^k) | 1 - e^{(A+\sigma)(t_s^k - q(t_s^k))}
      | e^{A\Delta'_k}$,  where we have used the fact $\Delta'_k = t_s^{k+1} - t_s^k$
  and~\eqref{eq:diffzzbar}. On simplification, we obtain
  \begin{align}\label{eq:h}
    \Delta'_k = \mathfrak{h}(t_s^k - q(t_s^k)) ,
  \end{align}
  \red{where, for convenience, we have defined $\mathfrak{h}(t):= - \frac{1}{A + \sigma} \ln ( | 1 - e^{(A+\sigma)t}| ) $}. Notice that $t_s^{k+1}$  depends only on $t_s^k$ and not on  $\Delta_k$ and. We show next that $t_s^k - q(t_s^k)$ uniquely determines $t_s^{k+1} - q(t_s^{k+1})$.

  To show this, recall that according to the proof of
  Theorem~\ref{thm:suf-cond-ET}, the quantization policy has the
  encoder divide the interval $[ \iota b \gamma, (\iota+1) b \gamma]$
  for some fixed $\iota$ uniformly into $2^{g-2}$ sub-intervals, one
  of which includes $t_s^k$. The decoder chooses as $q(t_s^k)$ the
  middle point of the sub-interval that contains $t_s^k$. Thus, we have
  \begin{equation}
    \label{eq:q}
    q(t) = \left \lfloor \frac{ t }{ \delta } \right \rfloor \delta +
    \frac{ \delta }{ 2 } , \quad \delta = \frac{ b \gamma }{ 2^{g-2} } .
  \end{equation}
  Letting $y_k = t_s^k - q(t_s^k)$, we obtain
  \begin{align*}
    y_{k+1}
    &= t_s^k + \Delta'_k - q(t_s^k + \Delta'_k)
    \\
    &= y_k + \left \lfloor \frac{ t_s^k }{ \delta } \right \rfloor \delta +
      \Delta'_k - \left \lfloor \frac{ y_k + \left \lfloor \frac{
      t_s^k }{ \delta } \right \rfloor \delta + \frac{ \delta }{ 2 } +
      \Delta'_k }{ \delta } \right \rfloor \delta
    \\
    &= y_k + \mathfrak{h}(y_k) - \left \lfloor \frac{ y_k + \frac{ \delta }{ 2 }
      + \mathfrak{h}(y_k) }{ \delta } \right
      \rfloor \delta =: \mathcal{H}(y_k) ,
  \end{align*}
  where in the second step we have used $t_s^k = y_k + q(t_s^k)$
  and~\eqref{eq:q}, and in the third step we have
  used~\eqref{eq:h}. From the conditions on $g$, we know
  that~\eqref{Quntrulesuff1} is satisfied and hence $\mathcal{H}$ is a map from
  the interval $[- \frac{ \delta }{ 2 }, \frac{ \delta }{ 2 } ]$ onto
  itself. We also notice that $\mathcal{H}$ is a piecewise continuous
  function. In fact, it is easy to verify that on
  $[- \frac{ \delta }{ 2 }, 0 )$, the function is piecewise strictly
  increasing. Further, note that if $\mathcal{H}$ is discontinuous at $w < 0$,
  then the left limit of $\mathcal{H}$ at $w$ is $\delta/2$ while the right
  limit of $\mathcal{H}$ at $w$ is $- \delta/2$.

  Next, \eqref{eq:expansion-cond} implies that
  \begin{equation*}
    \ln( 1 - e^{-(A+\sigma) \frac{\delta}{2}} ) - \ln( 1 - e^{-(A+\sigma)
      \frac{\delta}{4}} ) \geq (A+\sigma) \frac{3\delta}{4} ,
  \end{equation*}
  which, after rearranging the terms, we see that it implies
  \begin{equation*}
    -\frac{\delta}{4} + \mathfrak{h} \left( -\frac{\delta}{4} \right) \geq -
    \frac{\delta}{2} + \mathfrak{h} \left( - \frac{\delta}{2} \right) + \delta .
  \end{equation*}
  Now, observe that if $w_1, w_2 \in [- \frac{ \delta }{ 2 }, \frac{
    \delta }{ 2 } ]$ are such that $w_2 + \mathfrak{h}(w_2) = w_1 + \mathfrak{h}(w_1) + n
  \delta$ for some $n \in \integers$, then $\mathcal{H}(w_1) =
  \mathcal{H}(w_2)$. As a result, we conclude that there exists an
  interval $I \in [- \frac{ \delta }{ 2 }, -\frac{ \delta }{ 4 } ]$
  such that the restriction $\mathcal{H}: I \rightarrow [- \frac{
    \delta }{ 2 }, \frac{ \delta }{ 2 } ]$ is continuous, one-to-one
  and onto. Hence the inverse mapping of this restriction is
  continuous and is a contraction and hence using the Banach contraction
  principle~\cite{pugh2002real}, there exists a fixed point of the
  original map $\mathcal{H}$ in $I$. Finally, note that as we sweep
  $z(0)$ through $(0, v(0)]$, $t_s^1$ varies continuously from
  $\infty$ to $0$. Thus, there exists a $z(0)$ such that $y_1 = t_s^1
  - q(t_s^1)$ is the fixed point in $I$. This proves Claim~(b).
\end{proof}

\red{
\begin{remark}
{\rm We use the assumption in~\eqref{eq:expansion-cond} in the proof of Theorem~\ref{PROBDEF1} to be able to apply the Banach contraction principle in establishing the existence of a suitable initial condition. We use the assumption $\nu \ge 2$ to ensure that the upper bound in~\eqref{gtsupperbound} is well defined. 
}
  \oprocend
\end{remark}
}

\begin{figure} 
	%\centering
    	\includegraphics[height=7.5cm,width=8.2cm]{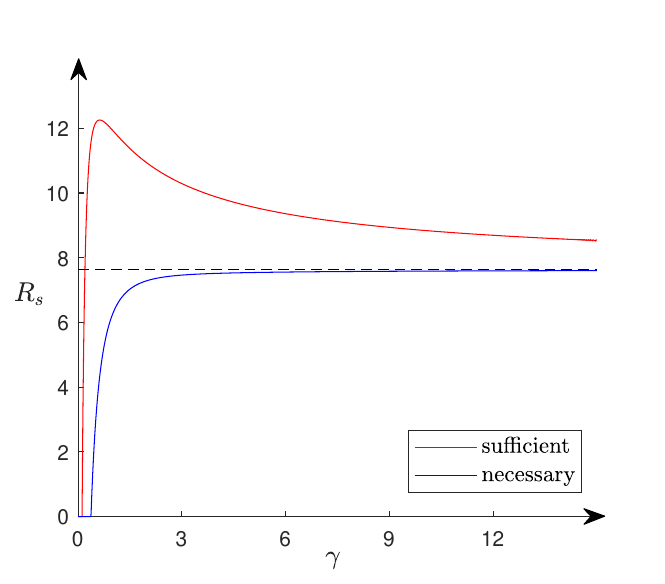}\caption{
    \red{Comparison between the sufficient and necessary conditions. $R_s$ is measured in $\mbox{bits}/\mbox{sec}$, and $\gamma$ is measure in $\mbox{sec}$.  Here, $A=1.3$, $\sigma=1$, $b=1.0001$, and $\rho_0=0.9$. The dashed line represents the asymptote $((A+\sigma)/\ln2)(1+A/\sigma)=7.6319$. 
    %Here, $A=1$, $\sigma=0.5$, $b=1.0001$, and $\rho_0=0.1$
    }}
    \label{fig:sn}
    \vspace*{-1ex}
\end{figure}

\red{
\begin{remark} {\rm Figure~\ref{fig:sn} illustrates the gap between
    the sufficient conditiont~\eqref{Sufi} and the supremum over $\sigma$ of the
    necessary condition~\eqref{necessT}. For
    small values of $\gamma$, both conditions reduce to $R_s>0$. As
    $\gamma$ grows to infinity, both conditions converge to the same
    asymptote with value
    $\frac{A+\sigma}{\ln2}(1+\frac{A}{\sigma})$.   While~\eqref{approxnec} reaches the asymptote  monotonically increasing for all $\rho_0$ values, the sufficient condition has an overshoot behavior for larger values of $\rho_0$ as depicted in Figure~\ref{fig:difrhosuff}. For intermediate values of $\gamma$, the gap can be explained
    noticing that  the exact value of the
    communication delay is unknown to the sensor and the
    controller, and hence there can be  a mismatch between the
    uncertainty sets at the controller and the sensor. In addition,  the    sensor and the controller lack a common reference frame for
    the quantization of the transmission time.
\oprocend}
\end{remark}

\begin{figure}
	\centering
\includegraphics[height=7.5cm,width=8.2cm]{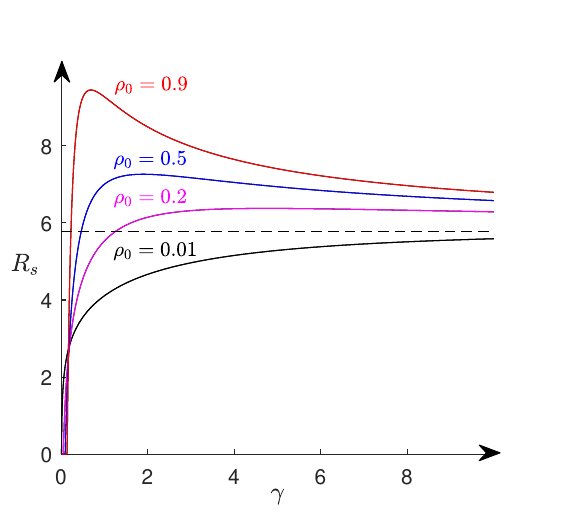}
     \caption{\red{Illustration of the sufficient transmission rate for asymptotic observability versus the upper bound of delay for different values of $\rho_0$. $R_s$ is measured in $\mbox{bits}/\mbox{sec}$, and $\gamma$ is measure in $\mbox{sec}$. Here, $A=1$, $\sigma=1$, and $b=1.0001$.  The dashed line represents the asymptote n$((A+\sigma)/\ln2)(1+A/\sigma)=5.7708$. 
     }}
     \label{fig:difrhosuff}
         \vspace*{-1ex}
 \end{figure}

\subsection{Simulation}
In this section, we illustrate an execution of our design for deriving the sufficient condition on the transmission rate. Using Theorem~\ref{thm:suf-cond-ET}, we choose the size of the packet
to be
\begin{align}\label{packetsimlatio1}
  g(t_s) = \max\left\{1,\lceil1+\log\frac{b\gamma
        (A+\sigma)}{\ln(1+\rho_0 e^{- (\sigma+A) \gamma})} \rceil\right\},
\end{align}
where the ceiling operator ensures that the packet size is an integer number (we take the maximum between this quantity and 1 to make sure to send at least one bit~\blue{of data payload} at each transmission). 

\begin{figure*}[htb]
  \centering
  \subfigure[]{\includegraphics[scale=0.67]{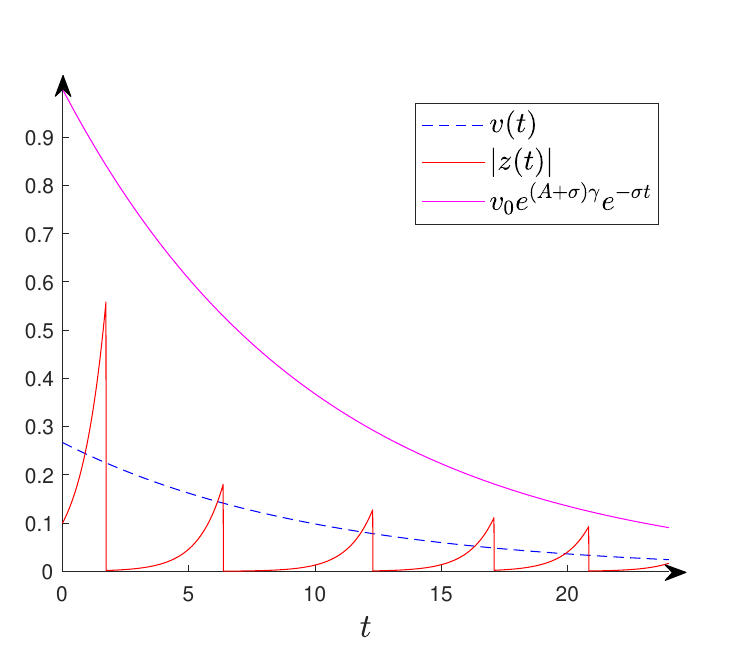}}
  \subfigure[]{\includegraphics[scale=0.6]{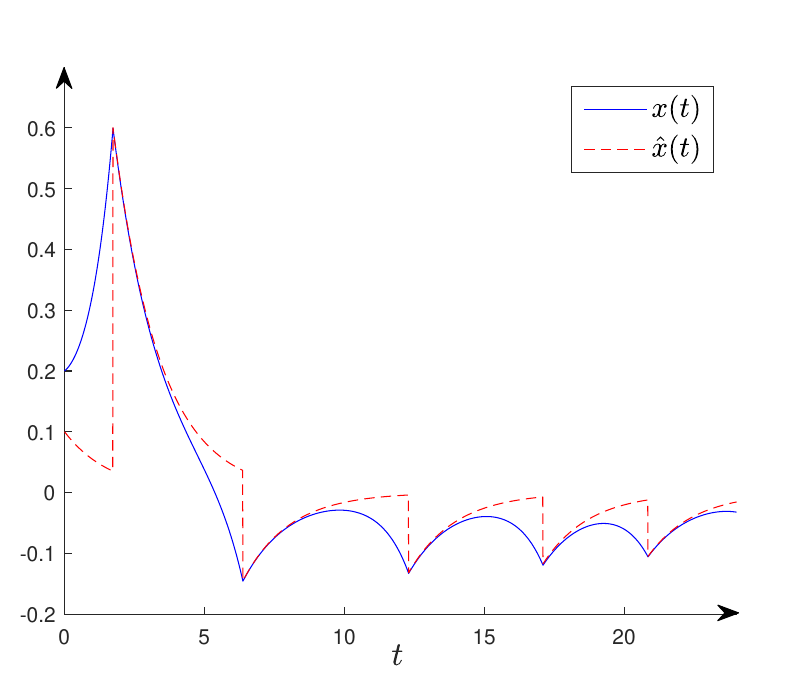}}
  \caption{\red{An example realization of our design. (a) shows the evolution of the absolute value of the state estimation error, the value of the event-triggering function, and the upper bound on the state estimation error. (b) shows the corresponding evolution of the state and state estimation. The continuous-time dynamics is discretized with step size $0.0002$.
  %Consequently, the minimum upper bound on the delay in communication process will be equal to the sampling time. 
Because of this, a triggering happens when $z(t)$ becomes larger than the triggering function and there is no packet in the communication channel. In fact, since the sampling time is small, a triggering happens when $z(t)$ becomes approximately equal $v(t)$.}}
  \label{fig:Actualc}
      \vspace*{-1ex}
\end{figure*}

We illustrate the execution of our design for the system
\begin{align*}
  \dot{x}(t)=x(t)+0.2u(t), \qquad u(t)=-8\hat{x}(t).
\end{align*}
%where.
The event-triggering function is $v(t)=0.2671 e^{-0.1t}$. The upper bound on the communication delay is  $\gamma=1.2$. The design parameter are $b=1.0001$, $\rho_0=0.1$, and the initial condition
 $x(0)=0.2$, and $\hat{x}(0)=0.1$.
Figure~\ref{fig:Actualc}(a) shows the evolution of the state estimation error. The triggering strategy ensures that
the state estimation error $z(t)$ converges exponentially to
zero and triggering occurs every time the state estimation error crosses
the triggering
function~$v(t)$. The overshoots observed in the
plot are due to the unknown delay in the communication channel. Clearly, $|z(t)|$ is upper bounded by $v_0 e^{(A+\sigma) \gamma} e^{-\sigma t}=e^{-0.1t}$. 
Figure~\ref{fig:Actualc}(b)  shows the corresponding evolution of $x(t)$ and $\hat{x}(t)$.
The values of $x(t)$ and $\hat{x}(t)$ become close to each other at the reception times because of the jump strategy, while the distance between $x(t)$ and $\hat{x}(t)$ grows during the inter-reception interval.

Finally, Figure~\ref{fig:ratesimulation} shows the information transmission rate of a simulation versus the delay upper bound $\gamma$ in the channel. The packet size is chosen according to~\eqref{packetsimlatio1}.  We calculate the information transmission rate by multiplying the packet size and the number of triggering events in the simulation time interval divided by its length. One can observe from the plot that, for small delay upper bound $\gamma$, the system is stabilized with an information transmission rate smaller than the data-rate theorem ($3.75$ bits$/$sec in this example). Instead, for larger $\gamma$, the transmission rate becomes greater than the threshold determined by the  data-rate theorem.

\begin{figure}
  \centering
  \includegraphics[scale=0.75]{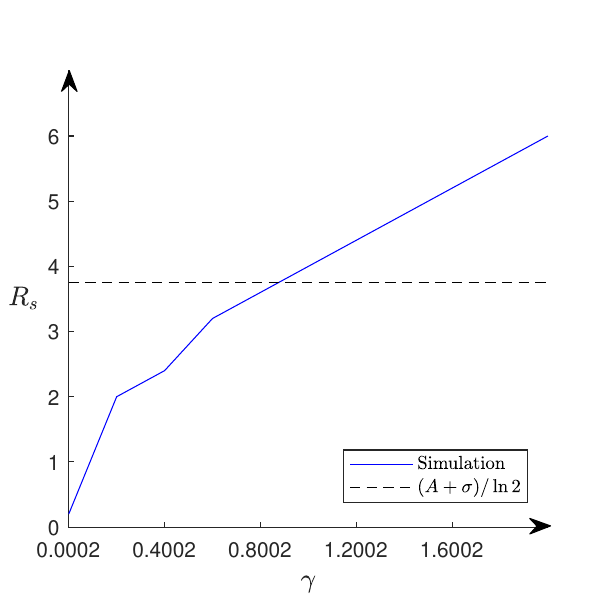}
  \caption{\red{Information transmission rate versus the upper bound of the delay in the communication channel.  $R_s$ is measured in $\mbox{bits}/\mbox{sec}$, and $\gamma$ is measured in $\mbox{sec}$. Here, $A=2.4$, $B=1$, $u(t)=-8\hat{x}(t)$, $\sigma=0.2$, $b=1.0001$, $\rho_0=0.1$, $v_0=0.0442$,
    $x(0)=0.201$, and $\hat{x}(0)=0.2$. The value of $\gamma$ ranges from $0.0005$ to $2.0005$, in steps of $0.2$.
    %(with the actual delay drawn uniformly in $[0,\gamma]$). 
   %The simulations are done for $\gamma$ with step size $0.2$. 
   For each value of $\gamma$, we compute the transmission rate over an interval of $7$ seconds of simulation.}
    }\label{fig:ratesimulation}
    \vspace*{-1ex}
\end{figure}
}

\section{Extension to vector systems}\label{sec:HigherDimensions}
We generalize here our results to vector systems, building on the scalar case.
% \subsection{Preliminaries}
Consider the plant-sensor-channel-controller tuple in
Figure~\ref{fig:system}, and let the plant dynamics be described by a
continuous-time, linear time-invariant (LTI) system
\begin{align}
  \label{sysconHD}
  \dot{x}=Ax(t)+Bu(t),
\end{align}
where $x(t) \in \real^n$ and $u(t) \in \real^m$ for $t \in [0,\infty)$
are the plant state and the control input, respectively. Here, $A \in
M_{n,n}(\mathbb{R})$, $B \in M_{n,m}(\mathbb{R})$, and $\norm{x(0)}<L$, where
$L$ is known to both sensor and controller. We assume all the
eigenvalues of $A$ are real.  Without loss of generality, we also
assume that they are positive (since stable modes do not need any
actuation and we can disregard them).  In this setting, the intrinsic
entropy rate of the plant is
\begin{align}
  h_v=\frac{\Tr(A)}{\ln 2} = \frac{\sum_{i=1}^n \lambda_i}{\ln 2}.
  \label{eq:inhHD}
\end{align}
Hence, to guarantee stability it is necessary for the controller to
have access to state information at a rate
\begin{align*}
%\label{DatarateHD}
  R_c \ge h_v.
\end{align*}
Using the Jordan block decomposition~\cite{Prasolov}, we can write the
matrix $A \in M_{n,n}(\mathbb{R})$ as $\Phi \Psi \Phi^{-1}$, where
$\Phi$ is a real-valued invertible matrix and
$\Psi=\text{diag}[J_1,\ldots, J_q]$, where each $J_j$ is a Jordan
block corresponding to the real-valued eigenvalue $\lambda_j$ of
$A$. We let $p_j$ indicate the order of each Jordan block. 
For simplicity of exposition, we assume from
here on that $A$ is equal to its Jordan block decomposition, that is,
$A=\text{diag}[J_1,\ldots, J_q]$.

In the following, we deal with each state coordinate separately. This
corresponds to treating the $n$-dimensional system as $n$ scalar,
coupled systems. When a triggering occurs for one of the coordinates,
the controller should be aware of which coordinate the received packet
corresponds to. Accordingly, we assume there are $n$ parallel
finite-rate digital communication channels between each coordinate of
the system and the controller, each subject to unknown, bounded delay.

We use the same notation of Section~\ref{sec:setup}, but add 
subindex $i$ and superindex $j$ to specify the $i^{\text{th}}$
coordinate of the $j^{\text{th}}$ Jordan block. So, for instance,
$\{ t_{s,i}^{k,j} \}_{k \in \integerspos}$,
$\{ t_{c,i}^{k ,j}\}_{k \in \integerspos}$, $g(t_{s,i}^{k,j})$ denote
the sequences of transmission times, reception times, and number of
bits that the sensor transmits at each triggering time. Similarly, the
$k^{\text{th}}$ \emph{communication delay} $ \Delta_{k,i}^j$ and
$k^{\text{th}}$ \emph{triggering interval} ${\Delta'}_{k,i}^{j}$ can
be specified for each coordinate. The communication delays for all
coordinates are uniformly upper-bounded by $\gamma$, a non-negative
real number known to both the sensor and the controller. 
The transmission rate for each coordinate is then
\begin{align*}
  R_{s,i} ^j= \limsup_{N_i^j \rightarrow \infty}
  \frac{\sum_{k=1}^{N_i^j} g(t_{s,i}^{k,j})}{\sum_{k=1}^{N_i^j}
    \Delta'^j_{k,i}}.
\end{align*}
Assuming $n$ parallel communication channels between the plant and the
controller, each devoted to a coordinate separately, we have
\begin{align*}
  R_s=\sum_{j=1}^{q}\sum_{i=1}^{p_j}R_{s,i}^j.
\end{align*}
Using the same notation of Section~\ref{sec:setup}, when
%When
referring to a generic triggering or reception time, we omit the
superscript~$k$. 

The controller maintains an estimate $\hat{x}$ of the state, which evolves
according to
%\begin{subequations}\label{eq:ET-strat}
\begin{align}\label{sysestHD}
  \dot{\hat{x}}(t)=A\hat{x}(t)+Bu(t),
\end{align}
during the inter-reception times. %, with $\hat{x}(0)=x_0$.
The \emph{state estimation error} is   $z(t)=x(t)-\hat{x}(t)$,
which initially is set to $z(0)=x(0)-x_0$.
For the $i^{\text{th}}$ coordinate of the $j^{\text{th}}$ Jordan block, we
consider an event-triggering function as in~\eqref{eq:etf} with
different initial values $v^j_0$ for each coordinate, namely
\begin{align}
  \label{eq:etfHD}
  v_i^j(t)=v_{0,i}^j e^{-\sigma t}.
\end{align}
For each coordinate, we employ the triggering rule~\eqref{eq:ets} and
the jump strategy~\eqref{eq:jumpst}.
When a triggering occurs for the $i^{\text{th}}$ coordinate of the
$j^{\text{th}}$ Jordan block, we assume that the sensor sends a packet large enough
to ensure
\begin{align}
  |z_i^j(t_{c,i}^{j+})| \le\rho_0  e^{-\sigma \gamma} v(t_{s,i}^{j}).
  \label{eq:jump-upp-HD}
\end{align}
When referring to a generic Jordan block, 
we omit the superscript and subscript~$j$.

Although each Jordan block is effectively independent of each other,
the vector case is not an immediate extension of the scalar
one. Specifically, from~\eqref{sysconHD} and~\eqref{sysestHD}, we have
that
\begin{align}\label{eq:errjordan}
  \dot{z}_1(t)&=\lambda z_1(t)+z_2(t)\\\nonumber & \; \vdots \\
  \nonumber \dot{z}_{p-1}(t)&=\lambda z_{p-1}(t)+z_p(t)\\\nonumber
  \dot{z}_{p}(t)&=\lambda z_{p}(t),
\end{align}
where $p$ denotes the order of the Jordan block. This shows that the
evolution of the coordinates is coupled and hence, even assuming
parallel communication channels, care must be taken in generalizing
the results for the scalar case.

Our first result generalizes Theorem~\ref{thm:necc-access-rate} on the
necessary condition for the information access rate.

\begin{theorem}\label{thm:necc-access-rate-HD}
  Consider the plant-sensor-channel-controller model described in
  Section~\ref{sec:setup}, with plant dynamics~\eqref{sysconHD},
  and state estimation error $z (t)$. Let $\sigma \in \real$ be
  positive, then the following necessary conditions hold:
  \begin{enumerate}
  \item If the state estimation error satisfies
    \begin{align*}
      \norm{z(t)} \le \norm{z(0)} ~e^{-\sigma t},
    \end{align*}
    then
    \begin{align*}
      b_c(t) \ge t~\frac{\Tr(A)+n\sigma}{\ln 2}+n~ \log
      \frac{L}{\norm{z(0)}}.
      %\label{necz2}
    \end{align*}

  \item If the system in~\eqref{sysconHD} is stabilizable and
    \begin{align*}
      \norm{x(t)} \le \norm{x(0)} ~e^{-\sigma t},
    \end{align*}
    then
    \begin{align*}
      b_c(t) \ge t ~\frac{\Tr(A)+n\sigma}{\ln 2} .
    \end{align*}
  \end{enumerate}
  In both cases, the information access rate is $R_c \ge \frac{\Tr(A)+n\sigma}{\ln 2}$.
\end{theorem}

The proof of this result, omitted for space reasons, is analogous to
that of Theorem~\ref{thm:necc-access-rate}, noting that for $A \in
M_{n,n}(\real)$ and $X \in \real^{n}$, $m(AX)=|\det(A)|m(X)$,
$\det(e^A)=e^{\Tr(A)}$, and that the Lebesgue measure of a sphere of
radius $\epsilon$ in $\mathbb{R}^n$ is $k_n \epsilon ^n$, where $k_n$
is a constant that changes with dimension.

We next generalize the necessary condition on the information
transmission rate. If $A$ is diagonalizable, then the necessary and
sufficient bit rate for the vector system is equal to the sum of the
necessary and sufficient bit rates that we provide in
Section~\ref{sec:estim-err} for each coordinate of the system. We now
generalize this idea to any matrix with real eigenvalues.

\begin{theorem}\label{thm:Necvector}
  Consider the plant-sensor-channel-controller model with plant
  dynamics~\eqref{sysconHD}, where all eigenvalues of $A$ are real,
  estimator dynamics~\eqref{sysestHD}, event-triggering
  strategy~\eqref{eq:ets}, event-triggering function~\eqref{eq:etfHD},
  and packet sizes such that $z_i^j(t_{c,i}^{k,j})$ is determined at
  the controller within a ball of radius $\rho(t_{s,i}^{k,j})=\rho_0
  e^{-\sigma \gamma} v(t_{s,i}^{k,j})$ with $\nu$-precision,
  ensuring~\eqref{eq:jump-upp-HD} via the jump
  strategy~\eqref{eq:jumpst} for all $ k \in \mathbb{N}$,
  $i=1,\ldots,p_j$, and $j=1,\ldots,q$.
  Then, there exists a delay realization and initial condition, such that
  \begin{align*}
    R_s \ge
    \sum_{j=1}^{q}\frac{p_j(\lambda_j+\sigma)}{\ln\nu+\ln(2+\frac{e^{\sigma
          \gamma}}{\rho_0})}\max \left\{0,\log \frac{(e^{\lambda_j
          \gamma}-1)}{ \rho_0 e^{-\sigma \gamma} }\right\}.
\end{align*}
\end{theorem}
\begin{proof}
  Since there is no coupling across
  different Jordan blocks in~\eqref{sysconHD}, the inherent entropy rate~\eqref{eq:inhHD}
  is
  \begin{align*}
    h_v(A)=h_v(J_1)+\cdots+h_v(J_q) .
  \end{align*}
  Therefore, it is enough to prove the result for one of the Jordan
  blocks. Let $J$ be a Jordan block of order $p$ with  associated
  eigenvalue $\lambda$.  Note that the part of the vector $z(t)$
  which corresponds to $J$ is governed by~\eqref{eq:errjordan}.
  The solution of the first differential equation
  in~\eqref{eq:errjordan} is
  \begin{align*}%\label{eq:hghcoup1}
    z_1(t)=e^{\lambda t}z_1(0)+e^{\lambda t} \int_{0}^{t} e^{-\lambda
      \tau} z_2(\tau)d\tau.
  \end{align*}
  % Due to the delay in the communication channel,
  If for the first coordinate a triggering event occurs at time
  $t_{s,1}$, then $z_1(t_{c,1})$ belongs to the set
  \begin{align*}
  % \label{ABCD}\nonumber
    & \Omega(z(t_{c,1}) | t_{s,1}) =\{y=y_1+y_2: y_1=\pm
    v_1(t_{s,1})e^{\lambda(t_{c,1}-t_{s,1})}, 
    \\
    & \hspace*{7ex} y_2=\int_{t_{s,1}}^{t_{c,1}}
    e^{\lambda(t_{c,1}-\tau)}z_2(\tau)d\tau;   \; t_{c,1} \in
    [t_{s,1},t_{s,1}+\gamma],
    \\
    & \hspace*{15ex} z_2(\tau)\in \zeta_\tau^{s,2} ~~ \text{for} ~\tau \in [t_{s,1},t_{c,1}]\},
  \end{align*}
  where $\zeta_\tau^{s,2}$ is the uncertainty set for $z_2(\tau)$ at the sensor.
  We define
  \begin{align*}
    Y_1&=\{y_1:y_1=\pm v(t_{s,1})e^{\lambda(t_{c,1}-t_{s,1})},t_{c,1} \in [t_{s,1},t_{s,1}+\gamma]\},
  \end{align*}
  which is the uncertainty set of $z_1(t_{c,1})$ given $t_{s,1}$ for the differential equation $\dot z_1 = \lambda z_1$.
  By comparing the definitions of the sets  $\Omega(z(t_{c,1}) | t_{s,1})$ and $Y_1$, we have
  \begin{align*}
    m(\Omega(z(t_{c,1}) | t_{s,1}))\ge m(Y_1).
  \end{align*}
  Finally, we apply Lemmas~\ref{thm:necc-cond-ET-g}
  and~\ref{thm:necc-cond-ET} for each coordinate separately, so that
  the necessary bit rate for each must satisfy
  \begin{align*}
    R_{s,i} \ge \frac{\lambda+\sigma}{\ln\nu+\ln(2+\frac{e^{\sigma
          \gamma}}{\rho_0})}\max \left\{0,\log \frac{(e^{\lambda
          \gamma}-1)}{ \rho_0 e^{-\sigma \gamma} }\right\}
  \end{align*}
  for $i=1,\ldots,p$. The result now follows. 
\end{proof}

Note that, when $\rho_0 \ll e^{\sigma \gamma}/\max\{2,\nu\}$, the
result in Theorem~\ref{thm:Necvector} can be simplified to
\begin{align*}
  R_s \ge \sum_{j=1}^{q}\frac{p_j (\lambda_j+\sigma)}{\ln 2} \max
  \left\{0 ,1+\frac{\log (e^{\lambda_j \gamma}-1)}{-\log (\rho_0
      e^{-\sigma \gamma})}\right\}.
\end{align*}
Our next result generalizes the sufficient condition of
Theorem~\ref{thm:suf-cond-ET} to vector systems.

\begin{theorem}\label{thm:Sufvector}
  Consider the plant-sensor-channel-controller model with plant
  dynamics~\eqref{sysconHD}, where all eigenvalues of $A$ are real,
  estimator dynamics~\eqref{sysestHD}, event-triggering
  strategy~\eqref{eq:ets}, and event-triggering
  function~\eqref{eq:etfHD}.  For the $j^{\text{th}}$ Jordan block
  choose the following sequence of design parameters
  \begin{align*}
    0<\rho_{1}^j<\ldots<\rho_{p_j-1}^j<\rho_{p_j}^j=\rho_0<1.
  \end{align*}
  If the state estimation error satisfies $|z_i^j(0)|\le v_{0,i}^j $,
  then we can achieve~\eqref{eq:jump-upp-HD} and
  \begin{align*}
    |z_{i}^j(t)|\le v_{0,i}^j
    ((\rho_0-\rho_{i}^j)+e^{(\lambda_j+\sigma)\gamma}) e^{-\sigma t}
  \end{align*}
  for $i=1,\ldots, p_j$ and $j=1,\ldots,q$, with an information
  transmission rate, $R_s$, at least equal to
  \begin{small}
    \begin{align*}
      \sum_{j=1}^{j=q}\sum_{i=1}^{i=p_j}\frac{
        (\lambda_j+\sigma)}{-\ln(\rho_0 e^{-\sigma \gamma})}
      \max\left(0,1+\log\frac{b\gamma
          (\lambda_j+\sigma)}{\ln(1+\rho_i^j e^{- (\sigma+\lambda_j)
            \gamma})}\right),
    \end{align*}
  \end{small}
  where
  \begin{align}\label{GRRMARTIN}
    0<v_{0,i}^j \le \frac{v_{0,i-1}^j
      (\lambda_j+\sigma)(\rho_0-\rho_{i}^j)}{((\rho_0-\rho_{i}^j)+e^{(\lambda_j+\sigma)\gamma})(e^{(\lambda_j+\sigma)\gamma}-1)}, 
  \end{align}
  % \end{small}
  for $i=2,\ldots,p_j$ and $j=1,\ldots,q$, and $b>1$.
\end{theorem}
\begin{proof}
  It is enough to prove the result for one Jordan block.  The solution
  of the last two equations in~\eqref{eq:errjordan} is
  \begin{align}\label{Solsuf}
    z_{p-1}(t)&=e^{\lambda t}z_{p-1}(0)+e^{\lambda t} \int_{0}^{t}
    e^{-\lambda \tau} z_p(\tau)d\tau,
    \\
    \nonumber {z}_{p}(t)&=e^{\lambda t}z_{p}(0).
  \end{align}
  The differential equation that governs $z_p(t)$ is similar to what
  we considered in Theorem~\ref{thm:suf-cond-ET}. It follows that if
  the transmission rate for coordinate $p$ is lower bounded
  as~\eqref{Sufi} and $|z_p(0)| \le v_{0,p}$, then we can ensure
  $|z_p(t)| \le v_{0,p} e^{(\sigma+\lambda)\gamma} e^{-\lambda t}$.

  Assume now that a triggering happens for coordinate $p-1$ at time $t_{s,p-1}$, namely $|z_{p-1}(t_{s,p-1})|=v(t_{s,p-1})$, and the controller receives the packet related to coordinate $p-1$
  at time $t_{c,p-1}$. 
  Then the uncertainty set  for $z_{p-1}(t_{c,p-1})$ at the
  controller is
  \begin{multline}\label{uncHDsuf}
    \Omega(z(t_{c,p-1})|
    t_{c,p-1})=\{w_{p-1}=w_{p-1}^{(1)}+w_{p-1}^{(2)}:
    \\
    w_{p-1}^{(1)}=\pm
    v_{p-1}(\bar{t}_{r,p-1})e^{\lambda(t_{c,p-1}-\bar{t}_{r,p-1})},
    \\
    w_{p-1}^{(2)}=\int_{\bar{t}_{r,p-1}}^{t_{c,p-1}}
    e^{\lambda(t_{c,p-1}-\tau)}z_p(\tau)d\tau;
    \\
    \bar{t}_{r,p-1} \in [t_{c,p}-\gamma,t_{c,p-1}],
    \\
    z_{p}(\tau) \in \zeta_\tau^{c,p} ~ \text{for} ~\tau \in
    [\bar{t}_{r,p-1},t_{c,p-1}]\},
  \end{multline}
  where $\zeta_\tau^{c,p}$ is the uncertainty set for $z_p(\tau)$ at
  the controller. Clearly, the measure of $\Omega(z(t_{c,p-1})|
  t_{c,p-1})$ is larger when $w_{p-1}^{(1)}$ and $w_{p-1}^{(2)}$
  in~\eqref{uncHDsuf} have the same sign.  Hence, we can assume that
  $z_{p-1}(\bar{t}_{r,p-1})$ and $z_{p}(\tau)$ for $\tau \in
  [\bar{t}_{r,p-1},t_{c,p-1}]$ and $\bar{t}_{r,p-1} \in
  [t_{c,p-1}-\gamma,t_{c,p-1}]$ are positive. Define
  \begin{multline*}
    W_{p-1}=\{w_{p-1}=w_{p-1}^{(1)}+w_{p-1}^{(2)}:
    \\
    w_{p-1}^{(1)}=\pm
    v_{p-1}(\bar{t}_{r,p-1})e^{A(t_{c,p-1}-\bar{t}_{r,p-1})},
    \\
    w_{p-1}^{(2)}=\int_{\bar{t}_{r,p-1}}^{t_{c,p-1}}
    e^{\lambda(t_{c,p-1}-\tau)}z_p(\tau)d\tau;
    \\
    \bar{t}_{r,p-1} \in [t_{c,p}-\gamma,t_{c,p-1}],
    \\
    |z_{p}(\tau)|\le v_{0,p} e^{(\sigma+\lambda)\gamma}
    e^{-\sigma\tau} ~ \text{for} ~ \tau \in
    [\bar{t}_{r,p-1},t_{c,p-1}]\}.
  \end{multline*}
  Clearly, we have
  \begin{align}
    m(\Omega(z(t_{c,p-1})| t_{c,p-1})) \le m(W_{p-1}).
    \label{by}
  \end{align}
  Hence, a sufficient condition for $W_{p-1}$ will also be a
  sufficient condition for $\Omega(z(t_{c,p-1})| t_{c,p-1})$.
  We note that $W_{p-1}$ is the Brunn-Minkowski sum of the following
  sets
  \begin{align*}
    W_{p-1}^{(1)}&=\{w_{p-1}^{(1)}:w_{p-1}^{(1)}=\pm v_{p-1}(\bar{t}_{r,p-1})e^{A(t_{c,p-1}-\bar{t}_{r,p-1})},\\
    &~~~~~~~~~~~~~~~~~~~~~~~~~~~~~~~\bar{t}_{r,p-1} \in [t_{c,p}-\gamma,t_{c,p-1}]\}\\
    W_{p-1}^{(2)}&=\{w_{p-1}^{(2)}:w_{p-1}^{(2)}=\int_{\bar{t}_{r,p-1}}^{t_{c,p-1}}
    e^{\lambda(t_{c,p-1}-\tau)}z_p(\tau)d\tau;
    \\
    &~~~|z_{p}(\tau)|\le v_{0,p} e^{(\sigma+\lambda)\gamma}
    e^{-\sigma\tau} ~ \text{for} ~ \tau \in
    [\bar{t}_{r,p-1},t_{c,p-1}],
    \\
    &~~~~~~~~~~~~~~~~~~~~~~~~~~~~~~~\bar{t}_{r,p-1} \in
    [t_{c,p}-\gamma,t_{c,p-1}] \}.
  \end{align*}
  By the Brunn-Minkowski inequality \cite{gardner2002brunn}, we have
  \begin{align}
    \label{bmi}
    m(W_{p-1})\ge m(W_{p-1}^{(1)})+m(W_{p-1}^{(2)}).
  \end{align}
  The operators in the definition of $W_{p-1}^{(1)}$ and
  $W_{p-1}^{(2)}$ are continuous and the operator in the definition of
  $W_{p-1}^{(2)}$ is integral.  Hence, even if during the time
  interval $[\bar{t}_{r,p-1},t_{c,p-1}]$ the value of $z_p(\tau)$
  jumps according to~\eqref{eq:jumpst}, $W_{p-1}^{(2)}$ remains a
  connected compact set. Therefore, $W_{p-1}^{(1)}$ and
  $W_{p-1}^{(2)}$ are closed intervals that are translation and
  dilation of each other. In this case, the inequality (\ref{bmi}) is
  tight~\cite{klain2011equality}, and by \eqref{by} we have
  \begin{align}\label{BRWE}
    m(\Omega(z(t_{c,p-1})| t_{c,p-1})) \le m(W_{p-1}^{(1)})+m(W_{p-1}^{(2)}).
  \end{align}
  This allows us to deal with each coordinate, $p-1$ and $p$,
  separately as follows.  If there is no coupling in the differential
  equation that governs $z_{p-1}(t)$, we have
  \begin{align*}
    \dot{z}_{p-1}(t)=\lambda z_{p-1}(t).
  \end{align*}
  Using Theorem ~\ref{thm:suf-cond-ET}, and equation~\eqref{BRWE}    with the rate
  \begin{small}
    \begin{align}\label{DSDJk}
      &R_{s,p-1} \ge 
      \\
      &\quad \frac{\lambda+\sigma}{-\ln(\rho_{p-1} e^{-\sigma \gamma})} \max\Big\{0,1+\log\frac{b\gamma (\lambda+\sigma)}{\ln(1+\rho_{p-1} e^{- (\sigma+\lambda) \gamma})}\Big\}, \nonumber
    \end{align}
  \end{small}
  we can ensure 
  \begin{align}\label{GGHHDD}
    % m(\Omega(z(t_{c,p-1}^+)| t_{c,p-1})) 
    % $\zeta_\tau^{c,p}$ is the uncertainty set for $z_p(\tau)$ at the controller. 
    \Upsilon_{t_{c,p-1}^+}^{c}
    \le \rho_{p-1} v_{p-1}(t_{c,p-1})+m(W_{p-1}^{(2)}),
  \end{align}
  where $\Upsilon_{t_{c,p-1}^+}^{c}$ is the uncertainty set for
  $z_{p-1}(t_{c,p-1}^+)$ at the controller.

  We now find an upper bound for $m(W_{p-1}^{(2)})$ as follows. Since,
  $R_{s,p}$ is lower bounded as~\eqref{Sufi}, we can ensure $|z_p(t)|
  \le v_{0,p} e^{(\sigma+\lambda)\gamma} e^{-\sigma t}$, and
  \begin{align}\label{5566A}\nonumber
    m(W_{p-1}^{(2)})&=\int_{t_{c,p-1}-\gamma}^{t_{c,p-1}} e^{\lambda(t_{c,p-1}-\tau)}z_p(\tau)d\tau\\\nonumber
    &\le v_{0,p} e^{(\sigma+\lambda)\gamma} e^{\lambda t_{c,p-1}}\int_{t_{c,p-1}-\gamma}^{t_{c,p-1}} e^{-(\lambda+\sigma)\tau} d\tau \\
    &=\frac{v_{0,p} e^{(\sigma+\lambda)\gamma} e^{-\sigma t_{c,p-1}}}{\lambda+\sigma}(e^{(\lambda+\sigma)\gamma}-1).
  \end{align}
  From~\eqref{GRRMARTIN}, we have 
  \begin{align*}%\label{assumcoup}
    v_{0,p} \le
    \frac{v_{0,p-1}(\lambda+\sigma)(\rho_0-\rho_{p-1})}{e^{(\lambda+\sigma)\gamma}(e^{(\lambda+\sigma)\gamma}-1)}.
  \end{align*}
  Hence,
  \begin{small}
    \begin{align*}
      \frac{v_{0,p} e^{(\sigma+\lambda)\gamma}e^{-\sigma t_{c,p-1}}}{\lambda+\sigma}(e^{(\lambda+\sigma)\gamma}-1) &\le (\rho_0-\rho_{p-1}) v_{0,p-1}e^{-\sigma t_{c,p-1}}\\
      &=(\rho_0-\rho_{p-1})v_{p-1}(t_{c,p-1}).
    \end{align*}
  \end{small}
  Consequently, from~\eqref{5566A}
  we have 
  \begin{align}\label{uppercoup}
    m(W_{p-1}^{(2)}) \le (\rho_0-\rho_{p-1})v_{p-1}(t).
  \end{align}
  Therefore, using~\eqref{GGHHDD} and~\eqref{uppercoup} we have
  $m(\Upsilon_{t_{c,p-1}^+}^{c}) \le \rho_0v_{p-1}(t_{c,p-1})$ and
  $ |z_{p-1}(t_c^+)| \le \rho_0v_{p-1}(t_{c,p-1})$.  When $R_{s,p}$ is
  lower bounded as~\eqref{Sufi} and $R_{s,p-1}$ is lower bounded
  as~\eqref{DSDJk}, we can ensure
  \begin{align*}
    |z_{p-1}(t)| \le ((\rho_0-\rho_{p-1})+e^{(\lambda+\sigma)\gamma})v_{p-1}(t_{c,p-1})
  \end{align*}
  because the solution of the differential equation that governs
  $z_{p-1}$ is given in~\eqref{Solsuf}, and using~\eqref{uppercoup} we
  have
  \begin{align*}
    &|z(t_{c,p-1})| \le \\ 
    & v_{p-1}(t_{c,p-1}-\gamma)e^{\lambda\gamma}+(\rho_0-\rho_{p-1})v_{p-1}(t_{c,p-1}) \\ 
    &=((\rho_0-\rho_{p-1})+e^{(\lambda+\sigma)\gamma})v_{p-1}(t_{c,p-1}).
  \end{align*}
  With the same procedure we can find the sufficient rate $R_{s,i}$
  for $i=p-2,\ldots,1$, and this concludes the proof.
\end{proof}

\begin{remark}
  {\rm In a Jordan block of order $p_j$, the
    inequality~\eqref{GRRMARTIN} provides an upper bound on the value of the triggering function for coordinate $i$ using the value of the triggering function for coordinate $i-1$, where
    $i=2,\ldots,p_j$.  This is a natural consequence of the coupling
    among the coordinates in a Jordan block, cf.~\eqref{eq:errjordan},
    which makes the error in coordinate $i$ affect the error in
    coordinates $1$ to $i-1$, for each $i=2,\ldots,p_j$.}  \oprocend
\end{remark}

Corollary~\ref{Stability} can be generalized, provided $(A,B)$ is
stabilizable, using a linear control $u(t)=-K\hat{x}(t)$ with $A-BK$
Hurwitz. This is a consequence of Theorem~\ref{thm:Sufvector} which
guarantees that, using the stated communication rate, the state
estimation error for each coordinate converges to zero exponentially
fast.

\red{
\begin{remark}
  {\rm In our discussion,  we have assumed that $\hat{x}(t)$ is known to both controller and sensor. Since the sensor has access to the state, using the system dynamics, it can deduce $u(t)$,
and then obtain $\hat{x}(t)$, cf.~\cite{sahai2006necessity}. 
Note that the controller design for our
sufficient condition is linear $u(t)=-K\hat{x}(t)$, and thus the sensor  can deduce $\hat{x}(t)$ assuming that $BK$ is \green{invertible. Alternatively, the controller can directly signal the acknowledgment of the reception of the packet (and as a result $t_c^k$) to the sensor by applying a control input to the system that excites a specific frequency of the state each time a symbol has been received, and the sensor can construct $\hat{x}(t)$ at all time $t$ if it knows the decoding rule at the controller.}
On the other hand, assuming knowledge of $\hat{x}(t)$ at the sensor does not affect the generality of the necessary condition.}  \oprocend
\end{remark}}

\section{Conclusions}\label{sec:conc}
We have studied event-triggered control strategies for stabilization
and exponential observability of linear plants in the presence of
unknown bounded delay in the communication channel between the sensor
and the controller. Our study has been centered on quantifying the value of the timing
information implicit in the triggering events. We have identified a necessary and a sufficient condition on the transmission rate required to guarantee stabilizability and observability of the system for a given event triggering strategy. Our results reveal a phase transition behavior as a function of the maximum delay in the communication channel, where for small delays, a positive transmission rate ensures the control objective is met, while for large delays, the necessary transmission rate is larger than that of classical data-rate theorems with periodic communication and no delay. Future research will consider disturbances to the plant dynamics, additional errors in the communication channel not caused by quantization, \green{extensions    to the case when the communication delay is a function of  thepacket size, replacing the Assumption $1$ with packet size constraints,} and the study of other event-triggering strategies.

\section*{Acknowledgements}
This research was partially supported by NSF award CNS-1446891.

\bibliography{mybib} 

\begin{thebibliography}{10}
\providecommand{\url}[1]{#1}
\csname url@rmstyle\endcsname
\providecommand{\newblock}{\relax}
\providecommand{\bibinfo}[2]{#2}
\providecommand\BIBentrySTDinterwordspacing{\spaceskip=0pt\relax}
\providecommand\BIBentryALTinterwordstretchfactor{4}
\providecommand\BIBentryALTinterwordspacing{\spaceskip=\fontdimen2\font plus
\BIBentryALTinterwordstretchfactor\fontdimen3\font minus
  \fontdimen4\font\relax}
\providecommand\BIBforeignlanguage[2]{{%
\expandafter\ifx\csname l@#1\endcsname\relax
\typeout{** WARNING: IEEEtran.bst: No hyphenation pattern has been}%
\typeout{** loaded for the language `#1'. Using the pattern for}%
\typeout{** the default language instead.}%
\else
\language=\csname l@#1\endcsname
\fi
#2}}

\bibitem{MJK-PT-JC-MF:16-allerton}
M.~J. Khojasteh, P.~Tallapragada, J.~Cort\'{e}s, and M.~Franceschetti, ``The
  value of timing information in event-triggered control: The scalar case,''
  \emph{Allerton Conference on Communication, Control, and Computing}, pp.
  1165--1172, Sept. 2016.

\bibitem{khojasteh2017time}
------, ``Time-triggering versus event-triggering control over communication
  channels,'' in \emph{IEEE Conference on Decision and Control}, Melbourne,
  Australia, Dec 2017, pp. 5432--5437.

\bibitem{kumar}
K.-D. Kim and P.~R. Kumar, ``Cyber-physical systems: A perspective at the
  centennial,'' \emph{Proceedings of the IEEE}, vol. 100~(Special Centennial
  Issue), pp. 1287--1308, 2012.

\bibitem{murray2003future}
R.~M. Murray, K.~J. Astrom, S.~P. Boyd, R.~W. Brockett, and G.~Stein, ``Future
  directions in control in an information-rich world,'' \emph{IEEE Control
  Systems}, vol.~23, no.~2, pp. 20--33, 2003.

\bibitem{Delchamps}
D.~F. Delchamps, ``Stabilizing a linear system with quantized state feedback,''
  \emph{IEEE Transactions on Automatic Control}, vol.~35, no.~8, pp. 916--924,
  1990.

\bibitem{wong1999systems}
W.~S. Wong and R.~W. Brockett, ``Systems with finite communication bandwidth
  constraints. {I}{I}. stabilization with limited information feedback,''
  \emph{IEEE Transactions on Automatic Control}, vol.~44, no.~5, pp.
  1049--1053, 1999.

\bibitem{baillieul1999feedback}
J.~Baillieul, ``Feedback designs for controlling device arrays with
  communication channel bandwidth constraints,'' in \emph{ARO Workshop on Smart
  Structures, Pennsylvania State Univ}, 1999, pp. 16--18.

\bibitem{Mitter}
S.~Tatikonda and S.~Mitter, ``Control under communication constraints,''
  \emph{IEEE Transactions on Automatic Control}, vol.~49, no.~7, pp.
  1056--1068, 2004.

\bibitem{nair2004stabilizability}
G.~N. Nair and R.~J. Evans, ``Stabilizability of stochastic linear systems with
  finite feedback data rates,'' \emph{SIAM Journal on Control and
  Optimization}, vol.~43, no.~2, pp. 413--436, 2004.

\bibitem{martins2006feedback}
N.~C. Martins, M.~A. Dahleh, and N.~Elia, ``Feedback stabilization of uncertain
  systems in the presence of a direct link,'' \emph{IEEE Transactions on
  Automatic Control}, vol.~51, no.~3, pp. 438--447, 2006.

\bibitem{Paolo}
P.~Minero, M.~Franceschetti, S.~Dey, and G.~N. Nair, ``Data rate theorem for
  stabilization over time-varying feedback channels,'' \emph{IEEE Transactions
  on Automatic Control}, vol.~54, no.~2, pp. 243--255, 2009.

\bibitem{Lorenzo}
P.~Minero, L.~Coviello, and M.~Franceschetti, ``Stabilization over {M}arkov
  feedback channels: the general case,'' \emph{IEEE Transactions on Automatic
  Control}, vol.~58, no.~2, pp. 349--362, 2013.

\bibitem{sukhavasi2016linear}
R.~T. Sukhavasi and B.~Hassibi, ``Linear time-invariant anytime codes for
  control over noisy channels,'' \emph{IEEE Transactions on Automatic Control},
  vol.~61, no.~12, pp. 3826--3841, 2016.

\bibitem{middleton2009feedback}
R.~H. Middleton, A.~J. Rojas, J.~S. Freudenberg, and J.~H. Braslavsky,
  ``Feedback stabilization over a first order moving average {G}aussian noise
  channel,'' \emph{IEEE Transactions on Automatic Control}, vol.~54, no.~1, pp.
  163--167, 2009.

\bibitem{ardestanizadeh2012control}
E.~Ardestanizadeh and M.~Franceschetti, ``Control-theoretic approach to
  communication with feedback,'' \emph{IEEE Transactions on Automatic Control},
  vol.~57, no.~10, pp. 2576--2587, 2012.

\bibitem{ding2016multiplicative}
J.~Ding, Y.~Peres, G.~Ranade, and A.~Zhai, ``When multiplicative noise stymies
  control,'' \emph{Annals of applied probability}, 2018, to appear.

\bibitem{de2005n}
C.~De~Persis, ``n-bit stabilization of n-dimensional nonlinear systems in
  feedforward form,'' \emph{IEEE Transactions on Automatic Control}, vol.~50,
  no.~3, pp. 299--311, 2005.

\bibitem{liberzon2009nonlinear}
D.~Liberzon, ``Nonlinear control with limited information,''
  \emph{Communications in Information \& Systems}, vol.~9, no.~1, pp. 41--58,
  2009.

\bibitem{topological}
G.~N. Nair, R.~J. Evans, I.~M. Mareels, and W.~Moran, ``Topological feedback
  entropy and nonlinear stabilization,'' \emph{IEEE Transactions on Automatic
  Control}, vol.~49, no.~9, pp. 1585--1597, 2004.

\bibitem{tatikonda2004stochastic}
S.~Tatikonda, A.~Sahai, and S.~Mitter, ``Stochastic linear control over a
  communication channel,'' \emph{IEEE transactions on Automatic Control},
  vol.~49, no.~9, pp. 1549--1561, 2004.

\bibitem{kostina2016ratemm}
V.~Kostina and B.~Hassibi, ``Rate-cost tradeoffs in control,'' in
  \emph{Allerton Conference on Communication, Control, and Computing}.\hskip
  1em plus 0.5em minus 0.4em\relax Monticello, IL: IEEE, 2016, pp. 1157--1164.

\bibitem{toli}
A.~Khina, Y.~Nakahira, Y.~Su, and B.~Hassibi, ``Algorithms for optimal control
  with fixed-rate feedback,'' in \emph{IEEE Conference on Decision and
  Control}, Melbourne, Australia, Dec 2017, pp. 6015--6020.

\bibitem{ling2010necessary}
Q.~Ling and H.~Lin, ``Necessary and sufficient bit rate conditions to stabilize
  quantized {M}arkov jump linear systems,'' in \emph{American Control
  Conference}, Baltimore, MD, 2010, pp. 236--240.

\bibitem{ranade2015control}
G.~Ranade and A.~Sahai, ``Control capacity,'' in \emph{Information Theory
  (ISIT), 2015 IEEE International Symposium on}.\hskip 1em plus 0.5em minus
  0.4em\relax IEEE, 2015, pp. 2221--2225.

\bibitem{nair2002communication}
G.~N. Nair, S.~Dey, and R.~J. Evans, ``Communication-limited stabilisability of
  jump {M}arkov linear systems,'' in \emph{15th Int. Symp. Mathematical Theory
  of Networks and Systems, Notre Dame, IN}, 2002.

\bibitem{liberzon2014finite}
D.~Liberzon, ``Finite data-rate feedback stabilization of switched and hybrid
  linear systems,'' \emph{Automatica}, vol.~50, no.~2, pp. 409--420, 2014.

\bibitem{yang2016finite}
G.~Yang and D.~Liberzon, ``Finite data-rate stabilization of a switched linear
  system with unknown disturbance,'' \emph{IFAC-PapersOnLine}, vol.~49, no.~18,
  pp. 1085--1090, 2016.

\bibitem{sahai2006necessity}
A.~Sahai and S.~Mitter, ``The necessity and sufficiency of anytime capacity for
  stabilization of a linear system over a noisy communication link. {P}art {I}:
  Scalar systems,'' \emph{IEEE Transactions on Information Theory}, vol.~52,
  no.~8, pp. 3369--3395, 2006.

\bibitem{matveev2009estimation}
A.~S. Matveev and A.~V. Savkin, \emph{Estimation and control over communication
  networks}.\hskip 1em plus 0.5em minus 0.4em\relax Springer Science \&
  Business Media, 2009.

\bibitem{minero2017anytime}
P.~Minero and M.~Franceschetti, ``Anytime capacity of a class of {M}arkov
  channels,'' \emph{IEEE Transactions on Automatic Control}, vol.~62, no.~3,
  pp. 1356--1367, 2017.

\bibitem{girish2013}
G.~Nair, ``A non-stochastic information theory for communication and state
  estimation,'' \emph{IEEE Transactions on Automatic Control}, vol.~58, pp.
  1497--1510, 2013.

\bibitem{Massimo}
M.~Franceschetti and P.~Minero, ``Elements of information theory for networked
  control systems,'' in \emph{Information and Control in Networks}.\hskip 1em
  plus 0.5em minus 0.4em\relax Springer, 2014, pp. 3--37.

\bibitem{Nair}
B.~G.~N. Nair, F.~Fagnani, S.~Zampieri, and R.~J. Evans, ``Feedback control
  under data rate constraints: An overview,'' \emph{Proceedings of the IEEE},
  vol.~95, no.~1, pp. 108--137, 2007.

\bibitem{Yukselbook}
S.~Y{\"u}ksel and T.~Ba{\c{s}}ar, \emph{Stochastic Networked Control Systems:
  Stabilization and Optimization under Information Constraints}.\hskip 1em plus
  0.5em minus 0.4em\relax Springer Science \& Business Media, 2013.

\bibitem{astrom2002comparison}
K.~J. Astrom and B.~M. Bernhardsson, ``Comparison of riemann and lebesgue
  sampling for first order stochastic systems,'' in \emph{IEEE Conference on
  Decision and Control}, vol.~2, Las Vegas, Nevada, USA, 2002, pp. 2011--2016.

\bibitem{Tabuada}
P.~Tabuada, ``Event-triggered real-time scheduling of stabilizing control
  tasks,'' \emph{IEEE Transactions on Automatic Control}, vol.~52, no.~9, pp.
  1680--1685, 2007.

\bibitem{WPMHH-KHJ-PT:12}
W.~P. M.~H. Heemels, K.~H. Johansson, and P.~Tabuada, ``An introduction to
  event-triggered and self-triggered control,'' in \emph{IEEE Conference on
  Decision and Control}, Maui, HI, 2012, pp. 3270--3285.

\bibitem{PT-JC:16-tac}
P.~Tallapragada and J.~Cort{\'e}s, ``Event-triggered stabilization of linear
  systems under bounded bit rates,'' \emph{IEEE Transactions on Automatic
  Control}, vol.~61, no.~6, pp. 1575--1589, 2016.

\bibitem{Level}
E.~Kofman and J.~H. Braslavsky, ``Level crossing sampling in feedback
  stabilization under data-rate constraints,'' in \emph{IEEE Conference on
  Decision and Control}, San Diego, CA, 2006, pp. 4423--4428.

\bibitem{ling2016bit}
Q.~Ling, ``Bit rate conditions to stabilize a continuous-time scalar linear
  system based on event triggering,'' \emph{IEEE Transactions on Automatic
  Control}, 2017, to appear.

\bibitem{pearson2017control}
J.~Pearson, J.~P. Hespanha, and D.~Liberzon, ``Control with minimal
  cost-per-symbol encoding and quasi-optimality of event-based encoders,''
  \emph{IEEE Transactions on Automatic Control}, vol.~62, no.~5, pp.
  2286--2301, 2017.

\bibitem{li2012stabilizing2}
L.~Li, X.~Wang, and M.~Lemmon, ``Stabilizing bit-rates in quantized event
  triggered control systems,'' in \emph{Proceedings of the 15th ACM
  international conference on Hybrid Systems: Computation and Control}.\hskip
  1em plus 0.5em minus 0.4em\relax ACM, 2012, pp. 245--254.

\bibitem{li2012stabilizingF}
------, ``Stabilizing bit-rate of disturbed event triggered control systems,''
  \emph{IFAC Proceedings Volumes}, vol.~45, no.~9, pp. 70--75, 2012.

\bibitem{quevedo2014stochastic}
D.~E. Quevedo, V.~Gupta, W.-J. Ma, and S.~Y{\"u}ksel, ``Stochastic stability of
  event-triggered anytime control,'' \emph{IEEE Transactions on Automatic
  Control}, vol.~59, no.~12, pp. 3373--3379, 2014.

\bibitem{demirel2013trade}
B.~Demirel, V.~Gupta, and M.~Johansson, ``On the trade-off between control
  performance and communication cost for event-triggered control over lossy
  networks,'' in \emph{Control Conference (ECC), 2013 European}.\hskip 1em plus
  0.5em minus 0.4em\relax IEEE, 2013, pp. 1168--1174.

\bibitem{tallapragada2016event}
P.~Tallapragada, M.~Franceschetti, and J.~Cort{\'e}s, ``Event-triggered
  second-moment stabilization of linear systems under packet drops,''
  \emph{IEEE Transactions on Automatic Control}, vol.~63, no.~8, pp.
  2374--2388, 2018.

\bibitem{linsenmayer2017delay}
S.~Linsenmayer, R.~Blind, and F.~Allg{\"o}wer, ``Delay-dependent data rate
  bounds for containability of scalar systems,'' \emph{IFAC-PapersOnLine},
  vol.~50, no.~1, pp. 7875--7880, 2017.

\bibitem{hespanha2002towards}
J.~Hespanha, A.~Ortega, and L.~Vasudevan, ``Towards the control of linear
  systems with minimum bit-rate,'' in \emph{Proc. 15th Int. Symp. on
  Mathematical Theory of Networks and Systems (MTNS)}, 2002.

\bibitem{liberzon2016entropy}
D.~Liberzon and S.~Mitra, ``Entropy and minimal data rates for state estimation
  and model detection,'' in \emph{Proceedings of the 19th International
  Conference on Hybrid Systems: Computation and Control}.\hskip 1em plus 0.5em
  minus 0.4em\relax ACM, 2016, pp. 247--256.

\bibitem{johansson1999regularization}
K.~H. Johansson, M.~Egerstedt, J.~Lygeros, and S.~Sastry, ``On the
  regularization of zeno hybrid automata,'' \emph{Systems \& control letters},
  vol.~38, no.~3, pp. 141--150, 1999.

\bibitem{pugh2002real}
C.~C. Pugh, \emph{Real mathematical analysis}.\hskip 1em plus 0.5em minus
  0.4em\relax Springer, 2002, vol. 2011.

\bibitem{Prasolov}
V.~V. Prasolov, \emph{Problems and theorems in linear algebra}.\hskip 1em plus
  0.5em minus 0.4em\relax American Mathematical Soc., 1994, vol. 134.

\bibitem{gardner2002brunn}
R.~Gardner, ``The {B}runn-{M}inkowski inequality,'' \emph{Bulletin of the
  American Mathematical Society}, vol.~39, no.~3, pp. 355--405, 2002.

\bibitem{klain2011equality}
D.~Klain, ``On the equality conditions of the {B}runn-{M}inkowski theorem,''
  \emph{Proceedings of the American Mathematical Society}, vol. 139, no.~10,
  pp. 3719--3726, 2011.

\end{thebibliography}
\bibliographystyle{IEEEtran}

\vspace*{-2ex}

\begin{IEEEbiography}
[{\includegraphics[height=1.25in,clip]{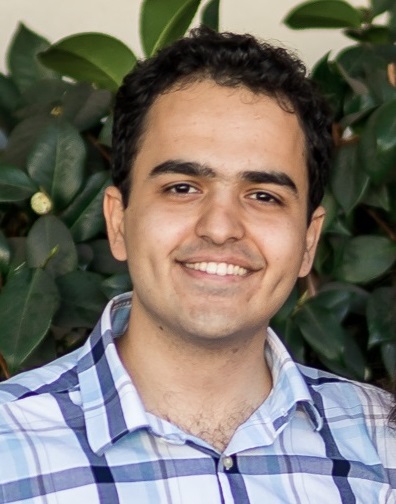}}]{Mohammad
    Javad Khojasteh}(S'14) did his undergraduate study at the Sharif University of Technology. He received two B.Sc. degrees in electrical engineering and pure mathematics in 2015.
    %form 2010 to 2015. 
    He received the
  M.Sc. degree in electrical and computer engineering at the
  University of California San Diego (UCSD), La Jolla, CA, in
  2017. Currently, he is pursuing a Ph.D. degree in electrical and
  computer engineering at UCSD under the advice of prof. Franceschetti. His research interests include networked control systems, machine learning, and robotics.
\end{IEEEbiography}
\begin{IEEEbiography}[{\includegraphics[height=1.25in,clip]
    {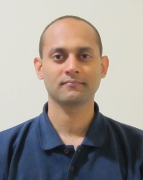}}]{Pavankumar Tallapragada}
  received the B.E. degree in Instrumentation Engineering from SGGS
  Institute of Engineering $\&$ Technology, Nanded, India in 2005,
  M.Sc. (Engg.) degree in Instrumentation from the Indian Institute of
  Science, Bangalore, India in 2007 and the Ph.D. degree in Mechanical
  Engineering from the University of Maryland, College Park in
  2013. He held a postdoctoral position at the University of
  California, San Diego during 2014 to 2017. He is currently an
  Assistant Professor in the Department of Electrical Engineering at
  the Indian Institute of Science, Bengaluru, India. His research
  interests include event-triggered control, networked control
  systems, distributed control and networked transportation systems.
\end{IEEEbiography}
\begin{IEEEbiography}
  [{\includegraphics[height=1.25in,clip]{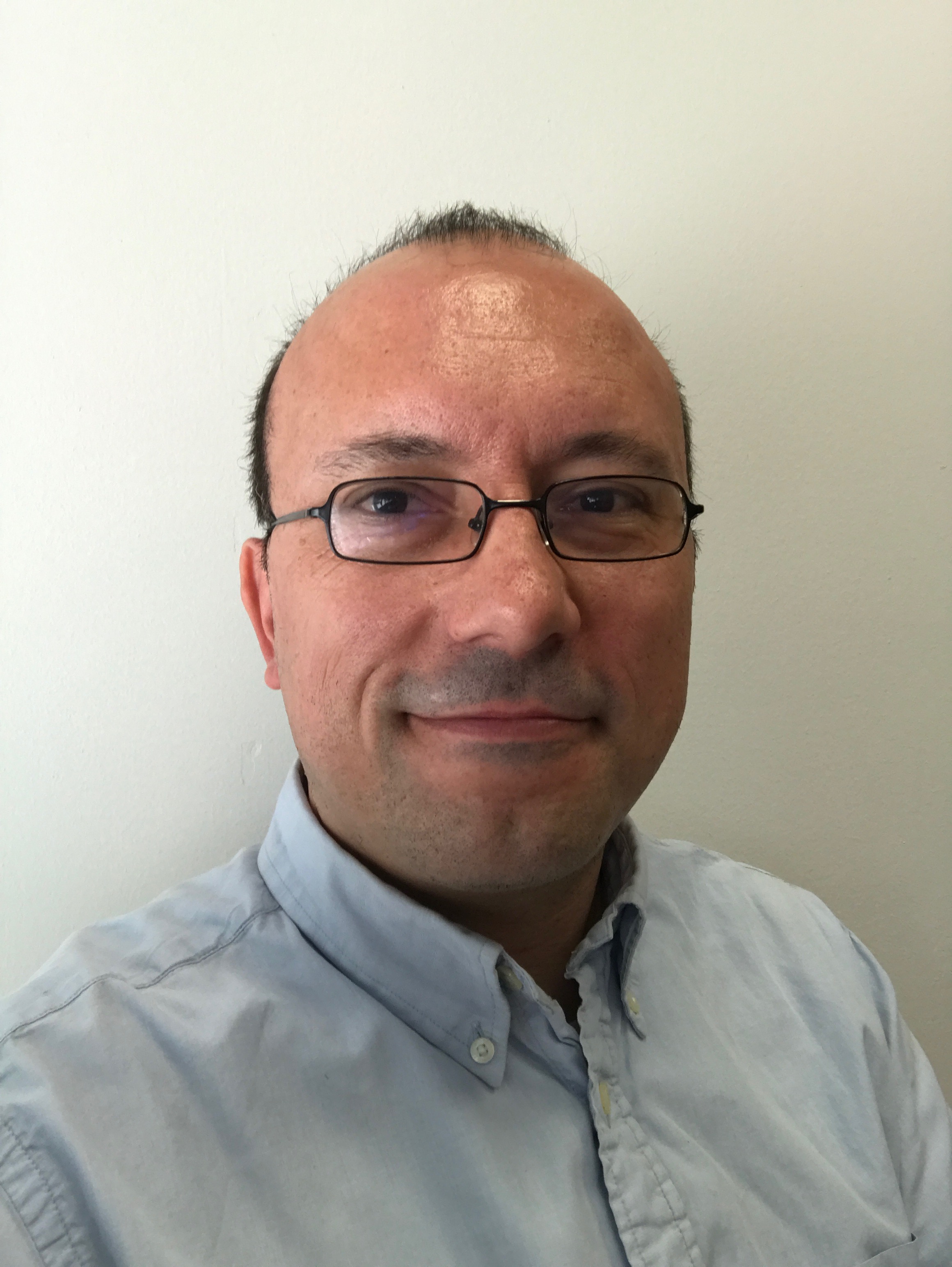}}]
  {Jorge Cort\'es} (M'02-SM'06-F'14) received the Licenciatura degree
  in mathematics from Universidad de Zaragoza, Zaragoza, Spain, in
  1997, and the Ph.D. degree in engineering mathematics from
  Universidad Carlos III de Madrid, Madrid, Spain, in 2001. He held
  postdoctoral positions with the University of Twente, Twente, The
  Netherlands, and the University of Illinois at Urbana-Champaign,
  Urbana, IL, USA. He was an Assistant Professor with the Department
  of Applied Mathematics and Statistics, University of California,
  Santa Cruz, CA, USA, from 2004 to 2007. He is currently a Professor
  in the Department of Mechanical and Aerospace Engineering,
  University of California, San Diego, CA, USA. He is the author of
  Geometric, Control and Numerical Aspects of Nonholonomic Systems
  (Springer-Verlag, 2002) and co-author (together with F. Bullo and
  S. Mart{\'\i}nez) of Distributed Control of Robotic Networks
  (Princeton University Press, 2009). He has been an IEEE Control
  Systems Society Distinguished Lecturer (2010-2014) and is an elected
  member for 2018-2020 of the Board of Governors of the IEEE Control
  Systems Society. His current research interests include distributed
  control and optimization, network science, opportunistic
  state-triggered control and coordination, reasoning under
  uncertainty, and distributed decision making in power networks,
  robotics, and transportation.
\end{IEEEbiography}
\begin{IEEEbiography}
  [{\includegraphics[height=1.25in,clip]{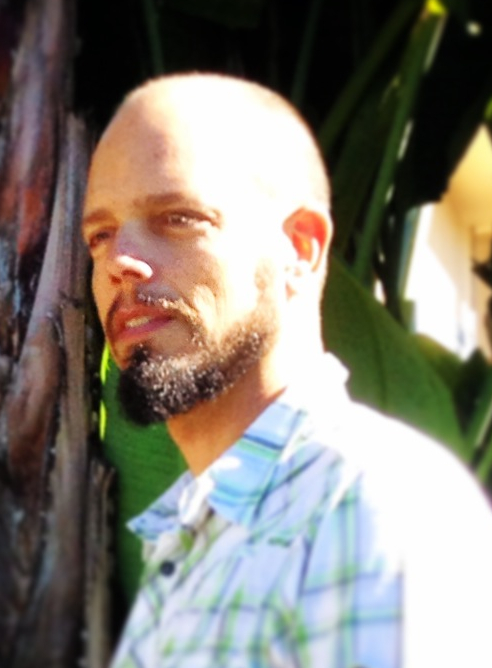}}]
  {Massimo Franceschetti} (M'98-SM'11-F'18) received the Laurea degree
  (with highest honors) in computer engineering from the University of
  Naples, Naples, Italy, in 1997, the M.S. and Ph.D. degrees in
  electrical engineering from the California Institute of Technology,
  Pasadena, CA, in 1999, and 2003, respectively.  He is Professor of
  Electrical and Computer Engineering at the University of California
  at San Diego (UCSD). Before joining UCSD, he was a postdoctoral
  scholar at the University of California at Berkeley for two
  years. He has held visiting positions at the Vrije Universiteit
  Amsterdam, the \'{E}cole Polytechnique F\'{e}d\'{e}rale de Lausanne,
  and the University of Trento. His research interests are in physical
  and information-based foundations of communication and control
  systems. He is a co-author of the book ``Random Networks for
  Communication'' and author of ``Wave Theory of Information,'' both published by Cambridge University Press.
  Dr. Franceschetti served as Associate Editor for Communication
  Networks of the IEEE Transactions on Information Theory (2009-2012), as Associate Editor of the IEEE Transactions on Control of
  Network Systems (2013-2016), as Associate Editor for the   IEEE Transactions on
  Network Science and Engineering (2014-2017), and as Guest Associate Editor of the IEEE
  Journal on Selected Areas in Communications (2008, 2009). He also served as general chair for the North American School of Information Theory (2015). He was awarded the C. H. Wilts
  Prize in 2003 for best doctoral thesis in electrical engineering at
  Caltech; the S.A.  Schelkunoff Award in 2005 for best paper in the
  IEEE Transactions on Antennas and Propagation, a National Science
  Foundation (NSF) CAREER award in 2006, an Office of Naval Research
  (ONR) Young Investigator Award in 2007, the IEEE Communications
  Society Best Tutorial Paper Award in 2010, and the IEEE Control
  theory society Ruberti young researcher award in 2012. He has been elected fellow of the IEEE in 2018, and he was nominated a Guggenheim Fellow for natural sciences, engineering, in 2019.
\end{IEEEbiography}
\end{document}